\documentclass[11pt,reqno]{amsart}

\usepackage{amsmath}								
\usepackage{amssymb}
\usepackage{amsthm}
\usepackage{amscd}
\usepackage{amsfonts}
\usepackage{stmaryrd}

\usepackage{xy}

\usepackage{euler}

\usepackage[pdftitle={Tropical geometry of moduli spaces of weighted stable curves},pdfauthor={Martin Ulirsch},ocgcolorlinks,linkcolor=linkblue,citecolor=linkred,urlcolor=linkred]{hyperref}
\usepackage{color}
\definecolor{linkred}{rgb}{0.6,0,0}
\definecolor{linkblue}{rgb}{0,0,0.6}

\usepackage{tikz}									
\usetikzlibrary{matrix}

\usetikzlibrary{patterns}
\usetikzlibrary{matrix}
\usepackage{fullpage}
\usepackage{enumerate}

\newtheorem{theorem}{Theorem}[section]
\newtheorem{lemma}[theorem]{Lemma}
\newtheorem{proposition}[theorem]{Proposition}
\newtheorem{corollary}[theorem]{Corollary} 

\theoremstyle{definition}
\newtheorem{definition}[theorem]{Definition}
\newtheorem{example}[theorem]{Example}

\theoremstyle{remark}
\newtheorem{remark}[theorem]{Remark}
\newtheorem{remarks}[theorem]{Remarks}

\newcommand{\R}{\mathbb{R}}
\newcommand{\Rbar}{\overline{\mathbb{R}}}

\newcommand{\ZZ}{\mathbb{Z}}
\newcommand{\Q}{\mathbb{Q}}

\newcommand{\N}{\mathbb{N}}
\newcommand{\PP}{\mathbb{P}}

\newcommand{\QQ}{\mathbb{Q}}

\newcommand{\G}{\mathbb{G}}

\newcommand{\Sigmabar}{\overline{\Sigma}}
\newcommand{\sigmabar}{\overline{\sigma}}
\newcommand{\Mbar}{\overline{M}}
\newcommand{\calA}{\mathcal{A}}
\newcommand{\calB}{\mathcal{B}}
\newcommand{\calC}{\mathcal{C}}
\newcommand{\calE}{\mathcal{E}}
\newcommand{\calG}{\mathcal{G}}
\newcommand{\calH}{\mathcal{H}}
\newcommand{\calM}{\mathcal{M}}

\newcommand{\calO}{\mathcal{O}}
\newcommand{\calS}{\mathcal{S}}
\newcommand{\calX}{\mathcal{X}}
\newcommand{\calMbar}{\overline{\mathcal{M}}}
\newcommand{\frakS}{\mathfrak{S}}
\newcommand{\frako}{\mathfrak{o}}
\newcommand{\scrD}{\mathscr{D}}
\newcommand{\scrW}{\mathscr{W}}

\DeclareMathOperator{\Spec}{Spec}
\DeclareMathOperator{\Hom}{Hom}

\DeclareMathOperator{\codim}{codim}
\DeclareMathOperator{\Aut}{Aut}

\DeclareMathOperator{\Isom}{Isom}
\DeclareMathOperator{\trop}{trop}
\DeclareMathOperator{\val}{val}
\DeclareMathOperator{\characteristic}{char}

\DeclareMathOperator{\Trop}{Trop}

\title[Tropical geometry of moduli spaces]{Tropical geometry of moduli spaces of weighted stable curves}
\author{Martin Ulirsch}
\address{Department of Mathematics, Brown University, Providence, RI 02912, USA}
\email{\href{mailto:ulirsch@math.brown.edu}{ulirsch@math.brown.edu}}
\urladdr{\href{http://www.math.brown.edu/~ulirsch/index.html}{http://www.math.brown.edu/~ulirsch/index.html}}
\date{\today}
\thanks{The author's research was supported in part by funds from BSF grant 201025 and NSF grants DMS0901278 and DMS1162367.} 

\begin{document}

\maketitle

\begin{abstract} 
Hassett's moduli spaces of weighted stable curves form an important class of alternate modular compactifications of the moduli space of smooth curves with marked points. In this article we define a tropical analogue of these moduli spaces and show that the naive set-theoretic tropicalization map can be identified with a natural deformation retraction onto the non-Archimedean skeleton. This result generalizes work of Abramovich, Caporaso, and Payne treating the Deligne-Knudsen-Mumford compactification of the moduli space of smooth curves with marked points. We also study tropical analogues of the tautological maps, investigate the dependence of the tropical moduli spaces on the weight data, and consider the example of Losev-Manin spaces.
\end{abstract}

\setcounter{tocdepth}{1}

\tableofcontents


\section{Introduction}

Throughout the article we work over an algebraically closed field $k$ that is endowed with the trivial norm. In \cite{Hassett_weightedmoduli} Hassett introduces a class of modular compactifications $\calMbar_{g,\calA}$ of the moduli space $\calM_{g,n}$ of smooth curves with $n$ marked points parametrized by an \emph{input datum} $(g,\calA)$ consisting of a non-negative integer $g$ together with a collection $\calA=(a_1,\ldots,a_n)$ of \emph{weights} $a_i\in\QQ\cap(0,1]$ such that 
\begin{equation*}
2g-2+a_1+\dots+a_n>0 \ . 
\end{equation*}

The moduli space $\calMbar_{g,\calA}$  parametrizes curves $(C,p_1,\ldots,p_n)$ with $n$ marked non-singular points on $C$ that are \emph{stable of type $(g,\calA)$}, i.e. nodal curves $(C,p_1,\ldots,p_n)$ with $n$ marked non-singular points that fulfill the following two conditions:
\begin{enumerate}
\item\label{item_amplecanonical} The twisted canonical divisor $K_C+a_1p_1+\ldots+a_np_n$ is ample.
\item A subset $p_{i_1},\ldots, p_{i_k}$ of the marked points is allowed to coincide only if the inequality $a_{i_1}+\ldots+a_{i_k}\leq 1$ holds. 
\end{enumerate}

In the case $(a_1,\ldots,a_n)=(1,\ldots,1)$ this condition is nothing but the traditional notion of an $n$-marked stable curve and so the compactification $\calMbar_{g,\calA}$ is exactly the well-known Deligne-Knudsen-Mumford compactification $\calMbar_{g,n}$ of $\calM_{g,n}$ introduced in \cite{DeligneMumford_moduliofcurves} and \cite{Knudsen_projectivityII}. 
 
In \cite[Theorem 2.1]{Hassett_weightedmoduli} Hassett shows that the moduli spaces $\calMbar_{g,\calA}$ are connected Deligne-Mumford stacks that are proper and smooth over $\Spec \ZZ$ and whose coarse moduli spaces $\Mbar_{g,\calA}$ are projective over $\Spec\ZZ$. Denote by $\calM_{g,\calA}$ the open locus of smooth curves in $\calMbar_{g,\calA}$. The following Theorem \ref{thm_SNC} is well-known to the experts, but, to the best of the author's knowledge, it has not appeared in the literature, so far. A discussion of its proof can be found in Section \ref{section_alternativeSNC}.

\begin{theorem}\label{thm_SNC}
The complement of $\calM_{g,\calA}$ in $\calMbar_{g,\calA}$ is a divisor with (stack-theoretically) normal crossings.
\end{theorem}

So the open immersion $\calM_{g,\calA}\hookrightarrow \calMbar_{g,\calA}$ has the structure of a toroidal embedding  (see \cite{KKMSD_toroidal}). By the work of \cite{Thuillier_toroidal} and \cite{AbramovichCaporasoPayne_tropicalmoduli},  associated to this datum there is a natural strong deformation retraction $\mathbf{p}$ from the non-Archimedean analytic space $\Mbar_{g,\calA}^{an}$ associated to the coarse moduli space $\Mbar_{g,\calA}$ onto a closed subset $\mathfrak{S}(\calMbar_{g,\calA})$ of $\Mbar_{g,\calA}^{an}$, called the  \emph{skeleton} of $\calMbar_{g,\calA}$, and this skeleton naturally carries the structure of an extended generalized cone complex in the sense of \cite[Section 2]{AbramovichCaporasoPayne_tropicalmoduli}.

In this article we define a notion of \emph{stability of type $(g,\calA)$} for tropical curves $\Gamma$ by imitating condition \eqref{item_amplecanonical}. Moreover we construct a set-theoretic moduli space $M_{g,\calA}^{trop}$ parametrizing isomorphism classes of tropical curves that are stable of type $(g,\calA)$. Its natural extension $\Mbar_{g,\calA}^{trop}$ admits an interpretation as a set-theoretic moduli space of extended tropical curves that are stable of type $(g,\calA)$. The tropical moduli space $\Mbar_{g,\calA}^{trop}$ naturally carries the structure of a generalized extended cone complex.

Moreover, following \cite{BakerPayneRabinoff_nonarchtrop}, \cite[Section 2.2.3]{Viviani_tropcompTorelli}, and \cite[Section 1.1]{AbramovichCaporasoPayne_tropicalmoduli}, there is a naive set-theoretic \emph{tropicalization map}
\begin{equation*}
\trop_{g,\calA}\mathrel{\mathop:}\Mbar_{g,\calA}^{an}\longrightarrow \Mbar_{g,\calA}^{trop}
\end{equation*}
from the non-Archimedean analytic space $\Mbar_{g,\calA}^{an}$ onto the extended tropical moduli space $\Mbar_{g,\calA}^{trop}$ defined as follows:

A point $x$ in $\Mbar_{g,\calA}^{an}$ can be represented by a morphism $\Spec K\rightarrow \calMbar_{g,\calA}$ for a non-Archimedean field extension $K$ of $k$. Since $\calMbar_{g,\calA}$ is a proper algebraic stack over $k$, the valuative criterion for properness implies that after a finite base change $K'\vert K$ this morphism extends uniquely to a morphism $\Spec R'\rightarrow\calMbar_{g,\calA}$, where $R'$ denotes the valuation ring of $K'$. This datum is equivalent to a curve $\calC\rightarrow \Spec R'$ that is stable of type $(g,\calA)$. Denote by $G_x$ the weighted dual graph of the special fiber $\calC_s$ of $\calC$; it is stable of type $(g,\calA)$ by Proposition \ref{prop_Astability}. At a node $p_e$ of $\calC_s$ corresponding to an edge $e$ of $G_x$ the curve is defined by $xy=f_e$ in formal coordinates, where $f_e\in R'$. Endowing an edge $e$ with the length $l(e)=\val(f_e)$, where $\val$ denotes the valuation on $R'$, defines a tropical curve $\Gamma_x$ and we set
\begin{equation*}
\trop_{g,\calA}(x)=[\Gamma_x]\in\Mbar_{g,\calA}^{trop}. 
\end{equation*}

The main result of this article can now be stated as follows:

\newpage

\begin{theorem}\label{thm_modulartrop}
There is a natural isomorphism $J_{g,\calA}\mathrel{\mathop:}\Mbar_{g,\calA}^{trop}\xrightarrow{\sim}\frakS(\calMbar_{g,\calA})$ of extended generalized cone complexes such that the diagram
\begin{center}\begin{tikzpicture}
  \matrix (m) [matrix of math nodes,row sep=2em,column sep=3em,minimum width=2em]
  {  
  & \Mbar_{g,\calA}^{an} & \\ 
 \frakS(\calMbar_{g,\calA})  & & \Mbar_{g,\calA}^{trop}  \\ 
  };
  \path[-stealth]
    (m-1-2) edge node [above left] {$\mathbf{p}$} (m-2-1)
    		edge node [above right] {$\trop_{g,\calA}$} (m-2-3)
    (m-2-3) edge node [below] {$J_{g,\calA}$} node [above] {$\sim$} (m-2-1);		
\end{tikzpicture}\end{center}
commutes.
\end{theorem}

In the case of $\calA=(1,\ldots,1)$ Theorem \ref{thm_modulartrop} is the same as \cite[Theorem 1.2.1]{AbramovichCaporasoPayne_tropicalmoduli}. Its proof, an adaption of the one in \cite{AbramovichCaporasoPayne_tropicalmoduli}, can be found in Section \ref{section_proof}. It relies on a careful analysis of the stratfication of $\calMbar_{g,\calA}$ by dual graphs, which is undertaken in Section \ref{section_dualgraphs&boundary}. In the case of $\calMbar_{g,n}$ this analysis can be found in \cite[Section XII.10]{ArbarelloCornalbaGriffiths_moduliofcurves}. Similar proofs come up in \cite{CavalieriMarkwigRanganathan_admissiblecovers} and in \cite{Ranganathan_ratcurvesnonArch} treating moduli spaces of admissible covers of the projective line and of rational logarithmic stable maps into a toric variety respectively.

From Theorem \ref{thm_modulartrop} we immediately obtain the following:

\begin{corollary}\label{cor_tropwelldefcont}
The naive set-theoretic tropicalization map $\trop_{g,\calA}\mathrel{\mathop:}\Mbar_{g,\calA}^{an}\rightarrow \Mbar_{g,\calA}^{trop}$ is well-defined and continuous.
\end{corollary}

The strong deformation retraction $\mathbf{p}$ restricts to a strong deformation retraction 
\begin{equation*}
M_{g,\calA}^{an}\longrightarrow\frakS(\calM_{g,\calA}) \ ,
\end{equation*}
 where $\frakS(\calM_{g,\calA})=M_{g,\calA}^{an}\cap\frak{S}(\calMbar_{g,\calA})$ is the \emph{non-Archimedean skeleton} of $\calM_{g,\calA}$. Theorem \ref{thm_modulartrop} therefore immediately implies the following:

\begin{corollary}\label{cor_modulartrop}
The isomorphism $J_{g,\calA}$ induces an isomorphism $M_{g,\calA}^{trop}\xrightarrow{\sim}\frakS(\calM_{g,\calA})$ of generalized cone complexes such that the diagram
\begin{center}\begin{tikzpicture}
  \matrix (m) [matrix of math nodes,row sep=2em,column sep=3em,minimum width=2em]
  {  
  & M_{g,\calA}^{an} & \\ 
 \frakS(\calM_{g,\calA})  & & M_{g,\calA}^{trop}  \\ 
  };
  \path[-stealth]
    (m-1-2) edge node [above left] {$\mathbf{p}$} (m-2-1)
    		edge node [above right] {$\trop_{g,\calA}$} (m-2-3)
    (m-2-3) edge node [below] {$J_{g,\calA}$} node [above] {$\sim$} (m-2-1);		
\end{tikzpicture}\end{center}
commutes. 
\end{corollary}

It is a natural question whether the tautological maps of the moduli spaces $\calMbar_{g,\calA}$, defined in analogy with \cite[Section 3]{Knudsen_projectivityII}, have tropical analogues and commute with tropicalization maps $\trop_{g,\calA}$. We give a positive answer to this Question in Section \ref{section_troptaut}. In particular we are going to focus on the forgetful map from \cite[Theorem 4.3]{Hassett_weightedmoduli} and the gluing and clutching maps as defined in \cite[Proposition 2.1.1]{BayerManin_weightedstablemaps} for the moduli spaces of weighted stable maps. 

Moreover, for two weight data $(g,\calA)$ and $(g,\calB)$ with $a_i\geq b_i$ for all $1\leq i\leq n$ Hassett constructs in \cite[Section 4]{Hassett_weightedmoduli} a proper birational \emph{reduction morphism} 
\begin{equation*}
\rho_{\calA,\calB}\mathrel{\mathop:}\calMbar_{g,\calA}\rightarrow\calMbar_{g,\calB}
\end{equation*}
that contracts those boundary divisors that parametrize $\calA$-stable curves that are not $\calB$-stable. These reduction morphisms play a central role in the birational geometry of the Deligne-Knudsen-Mumford moduli spaces $\calMbar_{g,n}$, since they form contractions onto certain log-canonical models of $\calMbar_{g,n}$. For these developments we refer the reader to \cite{Fedorchuk_weightedmoduli} and \cite{Moon_logcanonicalmodels} as well as to the survey \cite{FedorchukSmyth_alternatemoduli}. In Section \ref{section_forgetful&reduction} we define tropical analogues of these morphisms that commute with the tropicalization map. 

In Section \ref{section_wallsandchambers} we investigate how the tropical moduli spaces $\Mbar_{g,\calA}^{\trop}$ vary in the weights $\calA$ based on how the moduli spaces $\calMbar_{g,\calA}$ vary in the weights $\calA$. 
In Section \ref{section_LosevManin} we finish with the classical example of Losev-Manin spaces, as defined in  \cite{LosevManin_newmoduli}. 

During the work on this project the author learned about the article \cite{CavalieriHampeMarkwigRanganathan_weightedmoduli}, which contains the $g=0$ case of Theorem \ref{thm_modulartrop} in \cite[Theorem 3.15]{CavalieriHampeMarkwigRanganathan_weightedmoduli}. The main goal of \cite{CavalieriHampeMarkwigRanganathan_weightedmoduli}, however, is to treat the tropicalization of $\Mbar_{0,\calA}$ from the point of view of geometric tropicalization, as developed in \cite{HackingKeelTevelev_modulidelPezzo} and further studied in \cite{Cueto_geometrictropicalization}. For this the authors of \cite{CavalieriHampeMarkwigRanganathan_weightedmoduli} embed $\Mbar_{0,\calA}$ into a toric variety $X$ and study the tropicalization $\Trop_X(\Mbar_{0,\calA})$ of $\Mbar_{0,\calA}$ with respect to $X$. By \cite[Theorem 1.1 and 1.2]{Ulirsch_functroplogsch} as well as Theorem \ref{thm_modulartrop} there is a natural continuous and surjective map 
\begin{equation*}
\Mbar_{0,\calA}^{trop}\longrightarrow \Trop_X(\Mbar_{0,\calA})
\end{equation*}
that is, in general, not injective, as can be seen in \cite[Figure 4]{CavalieriHampeMarkwigRanganathan_weightedmoduli}. The main result of \cite{CavalieriHampeMarkwigRanganathan_weightedmoduli} is a characterization of those weights $\calA$ for which this map is a bjiection, i.e. for which the geometric tropicalization of $\Mbar_{0,\calA}$ faithfully represents the full tropical moduli space $\Mbar_{0,\calA}^{trop}$.

\subsection{Acknowledgements} 
The author would like to express his gratitude to Dan Abra\-movich for his constant support and encouragement. Thanks are also due to Renzo Cavalieri, Noah Giansiracusa, Simon Hampe, Diane MacLagan, Steffen Marcus, and Dhruv Ranganathan for several discussions related to this project. Finally, we would also like to thank the anonymous referee for many helpful remarks and suggestions. 

\section{Weighted stable tropical curves and their moduli}

In this section we define tropical versions of Hassett's moduli spaces $\calMbar_{g,\calA}$ of weighted stable curves. Our treatment of tropical moduli spaces in this section is strongly inspired by \cite[Section 3]{Caporaso_tropicalmoduli} and \cite[Section 4]{AbramovichCaporasoPayne_tropicalmoduli}. 

\subsection{$\mathcal{A}$-stability for weighted graphs}

Recall (e.g. \cite[Section 3.2]{AbramovichCaporasoPayne_tropicalmoduli}, \cite[Definition 2.1]{Caporaso_tropicalmoduli}, or \cite[Definition 2.3 and 2.4]{Manin_quantumbook}) that a \emph{weighted graph $G$ with $n$ marked  legs} is a sextuple
\begin{equation*}
\big(V(G),F(G),r,i,m,h\big) 
\end{equation*}
consisting of:
\begin{itemize}
\item a finite set of vertices $V(G)$,
\item a finite set of \emph{flags} $F(G)$ together with a \emph{root map} $r\mathrel{\mathop:}F(G)\rightarrow V(G)$ associating to a flag of $G$ the vertex it emanates from,
\item an involution $i\mathrel{\mathop:}F(G)\rightarrow F(G)$ inducing a decomposition of $F(G)$ into the set $L(G)$ of fixed points of $i$, called the \emph{legs} of $G$, and a finite union of pairs of points, called the \emph{edges} of $G$, 
\item a marking of $L(G)$, i.e. a bijection $m\mathrel{\mathop:}\{1,\ldots,n\}\xrightarrow{\sim} L(G)$ given by $L(G)=\{l_1,\ldots,l_n\}$, and
\item a \emph{weight function} $h\mathrel{\mathop:}V(G)\rightarrow \N$ associating to every vertex a nonnegative integer $h(v)$, referred to as the \emph{genus} of $v$.
\end{itemize}

We write $E(G)$ for the set of edges of $G$. Whenever there is no risk of confusion we sometimes drop the reference to $G$ from our notation and, for example, denote the set of vertices of $G$ by $V$ instead of $V(G)$.

The \emph{genus} $g(G)$ of $G$ is defined to be
\begin{equation}\label{eq_genusgraph}
g(G)=b_1(G)+\sum_{v\in V}h(v) \ ,
\end{equation}
where $b_1(G)=\dim_\Q H^1(G,\Q)=\#E(G)-\#V(G)+1$ is the first Betti number of $G$. An \emph{automorphism} $\gamma\in\Aut(G)$ consists of bijective maps $V(G)\xrightarrow{\sim} V(G)$ and $F(G)\xrightarrow{\sim} F(G)$ making the obvious diagrams commute. So, in particular, we have $g\big(\gamma(v)\big)=g(v)$ for all vertices $v\in V(G)$, the induced map $L(G)\xrightarrow{\sim}L(G)$ is the identity, and $\gamma$ preserves incidences between edge and vertices as well as legs and vertices. 

For a vertex $v\in V$, we write $L(v)$ for the set of marked legs emanating from $v$ and $\vert v\vert_E$ for the number of flags emanating from $v$ that are contained in an edge, i.e. the number of edges starting at $v$ counting loops with multiplicity two. Moreover, given an \emph{input datum} $(g,\calA)$, consisting of a non-negative integer $g$ together with a collection $\calA=(a_1,\ldots,a_n)$ of numbers $a_i\in\QQ\cap(0,1]$ such that 
\begin{equation*}
2g-2+a_1+\dots+a_n>0 \ ,
\end{equation*}
we set $\vert v\vert_\calA=\sum_{l_i\in L(v)}a_i$ for a vertex $v\in V$.

\begin{definition}\label{def_weightedstablegraph}
Let $(g,\calA)$ be an input datum. A weighted graph $G$ with $n$ legs is said to be \emph{stable of type $(g,\calA)$}, if it has genus $g$ and for all vertices $v\in V(G)$ we have:
\begin{equation*}
2h(v)-2+\vert v\vert_E+\vert v\vert_\calA> 0 \ .
\end{equation*}
If $G$ is stable of type $(g,\calA)$ with $\calA=(1,\ldots,1)$ we simply call it \emph{stable}.
\end{definition} 

 A \emph{weighted graph contraction} $\pi\mathrel{\mathop:}G\rightarrow G'$ is a composition of edge contractions that preserves the weight function in the sense that 
\begin{equation*}
h'(v')=g\big(\pi^{-1}(v')\big)
\end{equation*}
for all vertices $v'\in V'=V(G')$, where we consider $\pi^{-1}(v')$ as a subgraph of $G$. Observe that, if $G$ is stable of type $(g,\calA)$, the contracted graph $G'$ is stable of type $(g,\calA)$ as well. 

\begin{definition}
Let $(g,\calA)$ be an input datum. The category of $\mathcal{G}_{g,\calA}$ of weighted graphs that are stable of type $(g,\calA)$ is defined as follows:
\begin{itemize}
\item Its objects are isomorphism classes of weighted graphs $G$ that are stable of type $(g,\calA)$.
\item The morphisms in $\calG_{g,\calA}$ are generated by weighted graph contractions $\pi\mathrel{\mathop:}G\rightarrow G'$ and automorphisms $\Aut(G)$ of every graph $G$. 
\end{itemize}
\end{definition}

Note hereby that the set of isomorphisms between two weighted graphs $G_1$ and $G_2$ is either empty or a natural $\Aut(G_i)$-torsor for both $i=1,2$. In particular, a choice of an isomorphism $G_1\simeq G_2$ induces natural identifications between their automorphism groups $\Aut(G_i)$ and their weighted graph contractions $\pi_i\mathrel{\mathop:}G_i\rightarrow G_i'$.

\begin{remarks}
\begin{enumerate}[(i)]
\item The datum of the weight function $h\mathrel{\mathop:}V\rightarrow\N$ should be thought of as having  $h(v)$ infinitely small loops at each vertex $v$. This, in particular, explains why the genus of $G$ is defined as in \eqref{eq_genusgraph}.
\item Expanding on \cite{BakerNorine_RiemannRoch} we can define a $\QQ$-divisor on $G$ as a formal sum 
$\sum_{v\in V}a_v(v)$ with coefficients $a_v\in\QQ$. In this language a weighted graph $G$ is stable of type $(g,\calA)$ if and only if it has genus $g$ and the coefficients of the \emph{twisted canonical divisor} 
\begin{equation*}
K_{G,\calA}=\sum_{v\in V}\big(2h(v)-2+\vert v\vert_E+\vert v\vert_\calA\big)(v)
\end{equation*}
on $G$ are strictly positive. 
\item A theory similar to the one developed in this section has already been developed in \cite[Section 5]{BayerManin_weightedstablemaps}. We, in particular, refer the reader to \cite[Definition 5.1.6]{BayerManin_weightedstablemaps} which is equivalent to Definition \ref{def_weightedstablegraph}. Nevertheless note that the notion of a contraction in \cite[Definition 5.1.5]{BayerManin_weightedstablemaps} is different from ours, since the authors of \cite{BayerManin_weightedstablemaps} also allow legs to be merged.  
\end{enumerate}
\end{remarks}

\subsection{Tropical moduli spaces}
Following \cite[Definition 2.2]{Caporaso_tropicalmoduli} and \cite[Section 4.1]{AbramovichCaporasoPayne_tropicalmoduli} a \emph{tropical curve $\Gamma$ of genus $g$ with $n$ legs} consists of a weighted graph $G$ of genus $g$ with $n$ legs together with a \emph{length function} $l\mathrel{\mathop:}E(G)\rightarrow\R_{> 0}$. By identifying an edge $e$ with an interval of length $l(e)$ we can associate to $\Gamma$ a metric space $\vert\Gamma\vert$, which is called the \emph{geometric realization} of $\Gamma$. 

If we allow the length function $l$ to attain values in $\Rbar_{>0}=\R_{>0}\sqcup\{\infty\}$, we say that $\Gamma$ is an \emph{extended tropical curve} of genus $g$ with $n$ legs. Its \emph{geometric realization}  has the structure of an extended metric space by identifying an edge $e$ with $l(e)=\infty$ with the double infinite line \begin{equation*}
\big(\R_{\geq0}\sqcup\{\infty\}\big)\cup \big(\R_{\leq0}\sqcup\{-\infty\}\big) \ ,
\end{equation*}
where the two points $\infty$ and $-\infty$ are identified.

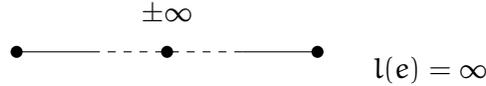
\begin{center}\begin{tikzpicture}
\fill (0,0) circle (0.08 cm);
\fill (-2,0) circle (0.08 cm);
\fill (2,0) circle (0.08 cm);
      
\draw (-2,0) -- (-1,0);
\draw (1,0) -- (2,0);
\draw [dashed] (-1,0) -- (0,0);
\draw [dashed] (1,0) -- (0,0);

\node at (0,0.5) {$\pm\infty$};
\node at (3.5,-0.25) {$l(e)=\infty$};
\end{tikzpicture}\end{center}

We denote the category of rational polyhedral cone complexes as defined in \cite[Section 3.2]{Ulirsch_functroplogsch} by $\mathbf{RPCC}$.

\begin{definition}
Let $(g,\calA)$ be an input datum. We define a natural contravariant functor 
\begin{equation*}
\Sigma\mathrel{\mathop:}\mathcal{G}_{g,\calA}\longrightarrow \mathbf{RPCC}
\end{equation*}
as follows:
\begin{itemize}
\item Associated to an isomorphism class of a weighted graph $G$ that is stable of type $(g,\calA)$ is the rational polyhedral cone $\sigma_G=\R_{\geq 0}^{E(G)}$.
\item A weighted edge contraction $\pi\mathrel{\mathop:}G\rightarrow G'$ induces the natural embedding $i_\pi\mathrel{\mathop:}\sigma_{G'}\hookrightarrow\sigma_G$ of a face of $\sigma_G$ .
\item An automorphism of $G$ induces an automorphism of $\sigma_G$.
\end{itemize}
\end{definition}
 
Similarly there is also a natural functor $\Sigmabar$ from $\mathcal{G}_{g,\calA}$ into the category of extended rational polyhedral cone complexes that is given by $G\mapsto \sigmabar_G=\Rbar_{\geq 0}^{E(G)}$. We denote the image of $\calG_{g,\calA}$ in $\mathbf{RPCC}$ simply by $\Sigma_{g,\calA}$ and  the category of its extensions by $\Sigmabar_{g,\calA}$.

\begin{definition}
The \emph{moduli space $M_{g,\calA}^{trop}$ of $\calA$-stable tropical curves of genus $g$ with $n$ marked legs} is defined to be the colimit
\begin{equation*}
M_{g,\calA}^{trop}=\lim_{\longrightarrow} \sigma_G \ ,
\end{equation*}
taken over $(\calG_{g,\calA})^{op}$. The \emph{moduli space $\Mbar_{g,\calA}^{trop}$ of extended $\calA$-stable tropical curves of genus $g$ with $n$ marked legs} is defined to be the colimit
\begin{equation*}
\Mbar_{g,\calA}^{trop}=\lim_{\longrightarrow} \sigmabar_G
\end{equation*}
taken over $(\calG_{g,\calA})^{op}$.
\end{definition}

The tropical moduli spaces $M_{g,\calA}^{trop}$ and $\Mbar_{g,\calA}^{trop}$ naturally carry the structure of a generalized cone complex and an extended generalized cone complex in the sense of \cite[Section 2]{AbramovichCaporasoPayne_tropicalmoduli} respectively.

Fix a weighted graph $G$ of genus $g$ with $n$ legs. We write $\sigma_G^\circ$ for the open cone $\R_{>0}^{E(G)}$. The set of tropical curves $\Gamma$ with underlying weighted graph isomorphic to $G$ can be parametrized by the quotient 
\begin{equation*}
M_G^{trop}=\sigma_G^\circ/\Aut(G) \ ,
\end{equation*}
the \emph{moduli space of tropical curve of combinatorial type $G$}. Similarly, if we replace $\sigma_G^\circ$ by the extended open cone $\sigmabar_G^\circ=\Rbar_{>0}^{E(G)}$, we obtain the set-theoretic \emph{moduli space 
\begin{equation*}
\Mbar_G^{trop}=\sigmabar_G^\circ/\Aut(G)
\end{equation*}
of extended tropical curves of combinatorial type $G$}.

From this point of view one can interpret the quotients $\sigma_G/\Aut(G)$ (or $\sigmabar_G/\Aut(G)$) as moduli spaces of tropical curves (or extended tropical curves), where we allow the edges to have zero length. For a weighted edge contraction $\pi\mathrel{\mathop:}G\rightarrow G'$ the faces $i_\pi\mathrel{\mathop:}\sigma_{G'}\hookrightarrow\sigma_G$ and $\overline{i}_\pi\mathrel{\mathop:}\sigmabar_{G'}\hookrightarrow\sigmabar_G$ parametrize those tropical curves, or extended tropical curves respectively, that have edge length zero for the edges that are collapsed by $\pi$.

As an immediate consequence of the construction we have:

\begin{proposition} 
There are decompositions
\begin{equation*}
M_{g,\calA}^{trop}=\bigsqcup_{G\in\calG_{g,\calA}}M_G^{trop}=\bigsqcup_{G\in\calG_{g,\calA}}\sigma^\circ_G/\Aut(G)
\end{equation*}
as well as 
\begin{equation*}
\Mbar_{g,\calA}^{trop}=\bigsqcup_{G\in\calG_{g,\calA}}\Mbar_G^{trop}=\bigsqcup_{G\in\calG_{g,\calA}}\sigmabar^\circ_G/\Aut(G) \ .
\end{equation*}
\end{proposition}


\section{Dual graphs and the boundary of $\overline{\mathcal{M}}_{g,\mathcal{A}}$}\label{section_dualgraphs&boundary}

We study the structure of the boundary $\calMbar_{g,\calA}-\calM_{g,\calA}$ of $\calMbar_{g,\calA}$, i.e. the complement of the locus of non-singular curves, using the machinery of dual graphs. The main result of this section is the proof of Theorem \ref{thm_SNC} which states that the boundary has (stack-theoretically) normal crossings. 

\subsection{Proof of Theorem \ref{thm_SNC}}\label{section_alternativeSNC} 
The author is aware of at least two different proofs of Theorem \ref{thm_SNC}.
Following \cite[Section 3.3]{Hassett_weightedmoduli} one can undertake a detailed analysis of the formal deformation spaces of weighted stable curves, using the deformation theory of maps, as developed in \cite{Ran_defofmaps}. This approach has been carried out in an earlier version of this article.

Here we present an alternative approach to the proof of Theorem \ref{thm_SNC} that is much shorter and more elementary. The author would like to thank Dan Abramovich for communicating this proof to him.
It essentially reduces Theorem \ref{thm_SNC} to the analogous theorem for the Deligne-Knudsen-Mumford moduli spaces $\calMbar_{g,n}$, a case that has already been discussed in \cite[Theorem 2.7]{Knudsen_projectivityII}. 

Set $N=\dim\calMbar_{g,\calA}=3g-3+n$. Let $\frako_k$ be either equal to $k$, if $\characteristic k=0$, or to the unique complete regular local ring with residue field $k$ and and maximal ideal generated by $p$, if $\characteristic k=p\neq 0$. Theorem \ref{thm_SNC} immediately follows from the following.

\begin{theorem}\label{thm_nodecoordinates}
Let $z$ be a point of $\calMbar_{g,\calA}$ corresponding to an $\calA$-stable curve $(C,p_1,\ldots,p_n)$ with nodes $x_1,\ldots, x_k$. Then there are formal coordinates $t_1,\ldots t_N$ around $z$ such that the complete local ring $\widehat{\calO}_{\calMbar_{g,\calA},z}$ is isomorphic to $\frako_k\llbracket t_1,\ldots,t_N\rrbracket$ and the locus where $x_i$ stays a node is given by $t_i=0$ for $1\leq i\leq k$. 
\end{theorem}

\begin{proof}
Denote by $\calS_{g,n}$ the algebraic stack of nodal curves of genus $g$ with $n$ (possibly singular) marked points, as introduced in \cite[Section 5]{Olsson_logtwistedcurves}. Note that $\calS_{g,n}$ is locally of finite type over $k$. Consider the universal curve $\calC_g$ of $\calS_g$ and denote by $\calC_g^{sm}$ the open substack of $\calC_{g}=\calS_{g,1}$ where the marked point is not singular. We can identifiy $\calMbar_{g,\calA}$ with an open substack of the $n$-fold fibered product $\calC_g^{sm}\times_{\calS_g}\cdots\times_{\calS_g}\calC_g^{sm}$, since $\calA$-stability is an open condition. Therefore the natural forgetful morphism $f\mathrel{\mathop:}\calMbar_{g,\calA}\rightarrow \calS_g$ is smooth.

Denote the image of $z$ in $\calS_g$ by $z'$ and let $N'=3g-3=N-n$. By \cite[Lemma 5.1]{Olsson_logtwistedcurves} there are formal coordinates $t'_1,\ldots,t'_{N'}$ around $z'$ such that
\begin{equation*}
\widehat{\calO}_{\calS_g,z'}\simeq \frako_k\llbracket t'_1,\ldots, t'_{N'}\rrbracket
\end{equation*}
and the locus where $x_i$ stays a node is given by $t'_i=0$ for $1\leq i\leq k$. Since $f\mathrel{\mathop:}\calMbar_{g,\calA}\rightarrow \calS_g$ is smooth, there are formal coordinates $t_1,\ldots, t_N$ around $z$ such that  
\begin{equation*}
\widehat{\calO}_{\calMbar_{g,\calA},z}\simeq \frako_k\llbracket t_1,\ldots, t_{N}\rrbracket
\end{equation*}
and $f^\ast t_i'=t_i$ for all $1\leq i\leq N'$. Therefore the locus where $x_i$ stays a node is given by $t_i=0$. 
\end{proof}

\begin{remark}\label{remark_nonNC}
In general, the complement of the smaller open subset $\calM_{g,n}\subseteq\calM_{g,\calA}$ in $\calMbar_{g,\calA}$ does not have normal crossings. This is due to the fact that marked points $p_{i_1},\ldots,p_{i_k}$ are allowed to coincide whenever the weights fulfill $a_{i_1}+\ldots+a_{i_k}\leq1$ without changing the combinatorial type of the curve. We refer the reader to the example of Losev-Manin spaces $L_n=\Mbar_{0,\calA}$ with $\calA=\big(1,\frac{1}{n},\ldots,\frac{1}{n},1\big)\in\Q^{n+2}\cap(0,1]^{n+2}$ discussed in Section \ref{section_LosevManin} below, where the complement of $M_{0,n+2}$ in $L_n$ clearly does not have normal crossings (see Example \ref{example_L_3nonSNC}).
\end{remark}

\subsection{Stratification by dual graphs}\label{section_stratification}

Let $(C,p_1,\ldots,p_n)$ be a complete and connected nodal curve of genus $g$ with $n$ marked points. We can associate to $C$ its \emph{dual graph} $G_C$, a weighted graph with $n$ marked legs that is defined as follows:
\begin{itemize}
\item The set $V=V(G)$ of vertices of $G$ is the set of irreducible components $C_i$ of $C$. 
\item The set of edges $E=E(G)$ is the set of nodes of $C$, where an edge $e$ connects two vertices $v_i$ and $v_j$ if and only if the corresponding components $C_i$ and $C_j$ meet each other in the node corresponding to $e$.
\item The set of legs $L=L(G)$ is the set of marked points of $C$. A leg $l_i$ emanates from a vertex $v$ if and only if the marked point $p_i$ corresponding to $l_i$ lies in the component $C_v$ corresponding to $v$.
\item The weight function $h\mathrel{\mathop:}V\rightarrow\N$ is defined by associating to a vertex $v$ the genus $g(C_v)$ of the corresponding component $C_v$.
\end{itemize}

It is well-known (see e.g. \cite[Proposition 2.6]{Manin_quantumbook}) that $g(C)=g(G_C)$ and that there is a natural homomorphism
\begin{equation*}
\Aut(C,p_1,\ldots,p_n)\longrightarrow \Aut(G_C)
\end{equation*}
of automorphism groups. 

\begin{proposition}\label{prop_Astability}
Let $(g,\calA)$ be an input datum. For a complete and connected nodal curve $(C,p_1,\ldots,p_n)$ of genus $g$ with $n$ marked points the following properties are equivalent:
\begin{enumerate}[(i)]
\item The twisted canonical divisor $K_C+a_1p_1+\ldots+a_np_n$ on $C$ is ample.
\item The dual graph $G_C$ of $C$ is stable of type $\calA$.
\end{enumerate}
\end{proposition}

\begin{proof}
The twisted canonical divisor $K_C+a_1p_1+\ldots+a_np_n$ on $C$ is ample if and only if its pullback to the normalization of each irreducible component of $C$ is effective and this is equivalent to 
\begin{equation*}
2h(v)-2+\vert v\vert_E+\vert v\vert_\calA> 0 
\end{equation*} 
for all vertices $v$ of $G_C$. But this is the exact definition of $G_C$ being stable of type $(g,\calA)$.
\end{proof}

Fix an input datum $(g,\calA)$. For a weighted graph $G$ of genus $g$ with $n$ legs that is stable of type $\calA$ we denote by $\calM_{G}$ the locally closed substack of $\calMbar_{g,\calA}$ consisting of those $\calA$-stable curves $C$ whose dual graph $G_C$ is equal to $G$. The closure of $\calM_G$ in $\calMbar_{g,\calA}$ will be denoted by $\calMbar_G$. Note that, if $G$ is the unique graph $\{\ast\}_{g,n}$ of genus $g$ with $n$ legs and one vertex, the stack $\calM_{\{\ast\}_{g,n}}$ exactly parametrizes those $\calA$-stable curves that are smooth and therefore coincides with $\calM_{g,\calA}$. 

It is not hard to see that the locally closed substacks $\calM_{G}$ are the strata of a stratification of $\calMbar_{g,\calA}$, i.e. 
\begin{equation*}
\calMbar_{g,\calA}=\bigsqcup_{G} \calM_G \ ,
\end{equation*}
where the disjoint union is taken over all isomorphism classes of weighted graphs $G$ of genus $g$ with $n$ legs that are stable of type $\calA$. As an immediate consequence of Theorem \ref{thm_SNC} we have:

\begin{corollary}\label{cor_codim=numedges}
The codimension of the locally closed stratum $\calM_G$ is equal to the number of edges of the $\calA$-stable weighted graph $G$.
\end{corollary}

\begin{proof}
The number $k=\#E(G)$ of edges of $G_C$ corresponds to the number of nodes of an $\calA$-stable curve $C$ in $\calMbar_{g,\calA}$. By Theorem \ref{thm_nodecoordinates} the complete local ring $\widehat{\calO}_{\calM_{g,\calA},x}$ at a point $x=\big[(C,p_1,\ldots,p_n)\big]$ of $\calM_G$ is equal to $\frako_k\llbracket t_1,\ldots,t_N\rrbracket$, where the coordinates $t_1,\ldots,t_N$ can be chosen such that the locus of the closure of $\calM_G$ is given by $t_1\cdots t_k=0$. This implies 
\begin{equation*}
\codim\calM_G=k=\#E(G)\ .
\end{equation*}
\end{proof}

An alternative proof of Corollary \ref{cor_codim=numedges} can be found at the end of Section \ref{section_algebraicclutching&gluing} below.

\begin{remark}
Stratifications of $\calMbar_{g,\calA}$ parametrized by the combinatorial data associated to weighted stable curves have already appeared in \cite{BayerManin_weightedstablemaps}, \cite{AlexeevGuy_weightedstablemaps}, and \cite{MustataMustata_weightedstablemaps}, all of which deal with moduli spaces of weighted stable maps. The crucial difference between their constructions and our approach is that their discrete data also contains information on whether marked points agree. Taking only dual graphs, we conveniently "forget" this information in order to obtain the stratification induced by the normal crossing boundary $\calMbar_{g,\calA}-\calM_{g,\calA}$.
\end{remark}

\subsection{Clutching and gluing}\label{section_algebraicclutching&gluing}
In this section we study analogues of the clutching and gluing morphisms originally defined in \cite[Section 3]{Knudsen_projectivityII} for $\calMbar_{g,n}$. This construction is a special case of  \cite[Proposition 2.1.1]{BayerManin_weightedstablemaps}, where these maps are defined for moduli spaces of weighted stable maps. 

For a vertex $v$ of $G$ we denote by $\calA(v)$ the tuple 
\begin{equation*}
(a_{i_1},\ldots a_{i_k},1,\ldots,1)
\end{equation*}
consisting of those $a_i$ that correspond to legs $l_i$ emanating from $v$ and a $1$ for every flag of an edge incident to $v$. Moreover we write $n_v$ for the number of entries of $\calA(v)$, i.e. the number of legs and edges emanating from $v$. 

\begin{proposition}\label{prop_clutching&gluing}
For every weighted graph $G$ that is stable of type $(g,\calA)$ there is a natural \emph{clutching and gluing morphism}
\begin{equation*}
\phi_G\mathrel{\mathop:}\prod_{v\in V(G)}\calMbar_{h(v),\calA(v)}\longrightarrow\calMbar_G\subseteq\calMbar_{g,\calA}
\end{equation*}
of the moduli stacks that associates to a tuple consisting of stable curves $(C^v,p_1^v,\ldots,p_{n_v}^v)$ of type $\big(h(v),\calA(v)\big)$ the stable curve $(C,p_1,\ldots, p_n)$ obtained by identifying two marked points, whenever they correspond to two flags defining an edge of $G$. 
\end{proposition}

\begin{proof}
Let $S$ be scheme and $(C^v)$  be a tuple of complete nodal curves over $S$ with sections $p_1^v,\ldots,p_n^v$ such that each $(C^v,p_1^v,\ldots,p_{n_v}^v)$ is stable of type $\big(h(v),\calA(v)\big)$. Then we can define a curve $C$ over $S$ by gluing the $C^v$ over two sections corresponding to two flags that are connected by an edge of $G$. Note that these sections do not intersect any other section, since they all have weight one and the $(C^v,p_1^v,\ldots,p_{n_v}^v)$ are stable of type $\big(h(v),\calA(v)\big)$. The resulting curve $(C,p_1,\ldots, p_n)$ over $S$ is stable of type $(g,\calA)$, since each $(C^v,p_1^v,\ldots,p_{n_v}^v)$ is stable of type $\big(h(v),\calA(v)\big)$ and the graph $\Gamma$ is stable of type $(g,\calA)$. This association commutes with arbitrary base changes $S'\rightarrow S$ and therefore defines a morphism of stacks. 
\end{proof}

\begin{corollary}\label{cor_clutching&gluing}\begin{enumerate}[(i)]
\item Suppose $(g_1,\calA_1)$ and $(g_2,\calA_2)$ are two weight data. Set $g=g_1+g_2$ and $\calA=\calA_1\cup\calA_2$. There is a natural \emph{clutching morphism} 
\begin{equation*}
\kappa=\kappa_{g_1,\calA_1,g_2,\calA_2}\mathrel{\mathop:} \calMbar_{g_1,\calA_1\sqcup\{1\}}\times\calMbar_{g_2,\calA_2\sqcup\{1\}}\longrightarrow \calMbar_{g,\calA}
\end{equation*}
that associates to a tuple consisting of the two $\calA_i\cup\{1\}$-stable curves $(C_i,p_1^i,\ldots,p_{n_i+1}^i)$ the $\calA$-stable curve $(C,p_1^1,\ldots, p_{n_1}^1,p_{1}^2,\ldots,p_{n_2}^2)$ obtained by identifying the two points $p_{n_1+1}^1$ and $p_{n_2+1}^2$ in a node. 
\item Fix an input datum $(g,\calA)$ with $g>0$. There is a natural \emph{gluing morphism} 
\begin{equation*}
\gamma\mathrel{\mathop:}\calMbar_{g-1,\calA\sqcup\{1,1\}}\longrightarrow\calMbar_{g,\calA}
\end{equation*} 
obtained by gluing together the last two marked points of an $\calA\cup\{1,1\}$-stable curve $(C,p_1,\ldots,p_{n+2})$ of genus $g-1$. 
\end{enumerate}\end{corollary}

\begin{proof}
Both Part (i) and Part (ii) are special cases of Proposition \ref{prop_clutching&gluing}. For Part (i) we take the graph $G$ to consist of two vertices $v_1$ and $v_2$ with weights $g_1$ and $g_2$ connected by an edge and having $n_1$ or $n_2$ legs incident to $v_1$ or $v_2$ respectively. For Part (ii) the graph $G$ consists only one vertex with weight $g$, from which $n$ legs are emanating, and a loop. 
\end{proof}

Set
\begin{equation*}
\widetilde{\calM}_G=\prod_{v\in V(G)}\calM_{h(v),\calA(v)} 
\end{equation*}
and note that the clutching and gluing morphism $\phi_G$ restricts to a morphism 
\begin{equation*}
\widetilde{\calM}_G\rightarrow \calM_G \ .
\end{equation*}

\begin{proposition}\label{prop_MGstructure}
For a weighted graph $G$ that is stable of type $(g,\calA)$ the clutching and gluing morphism $\phi_G\mathrel{\mathop:}\widetilde{\calM}_G\rightarrow \calM_G$ induces an isomorphism
\begin{equation*}
\Big[\widetilde{\calM}_G\big/\Aut(G)\Big]\simeq\calM_G  \ .
\end{equation*}
\end{proposition}

Our proof of Proposition \ref{prop_MGstructure} is a generalization of the proof of \cite[Proposition 10.11]{ArbarelloCornalbaGriffiths_moduliofcurves}.

\begin{proof}
We are going to show that both $\Big[\widetilde{\calM}_G\big/\Aut(G)\Big]$ and $\calM_G$ have the same groupoid presentation. 

By the construction in \cite[Example 8.15]{ArbarelloCornalbaGriffiths_moduliofcurves} we can find a $\Aut(G)$-invariant surjective \'etale morphism $s,t\mathrel{\mathop:}U\rightarrow \widetilde{\calM}_G$. In this case a groupoid presentation of $\Big[\widetilde{\calM}_G\big/\Aut(G)\Big]$ is given by $Y_1\rightrightarrows Y_0$, where 
\begin{equation*}
Y_1=\Aut(G)\times Y_0\times_{\calM_G}Y_0 \ .
\end{equation*}

The \'etale atlas $Y_0$ is \'etale locally isomorphic to a product $\prod_{v\in V(G)}U_v$, where $U_v$ are local slices of the 'exhausting family' of $\calM_{h(v),\calA(v)}$ around curves $C_v$ as introduced in \cite[Section 3.4]{Hassett_weightedmoduli}. The clutching and gluing map is induced by the morphism $\prod_{v\in V(G)} U_v\rightarrow U$ into a slice $U$ of the 'exhausting family' of $\calM_{g,\calA}$ around $C=\phi_G\big((C_v)_{v\in V(G)}\big)$ that is determined by isomorphically mapping $U_v$ to one of the branches of $U_G$, the locus in $U$ parametrizing curves with dual graph $G$. In particular the composition $Y_0\rightarrow\calM_G$ is surjective and \'etale.

The morphism $Y_0\rightarrow\widetilde{\calM}_G$ gives rise to a family $\eta\mathrel{\mathop:}\calC\rightarrow Y_0$ of curves with dual graphs equal to $G$. In this case we have natural isomorphisms
\begin{equation*}
Y_1=\Aut(G)\times Y_0\times_{\widetilde{\calM}_G} Y_0\simeq\Aut(G)\times \Isom_{Y_0\times Y_0}^G(s^\ast\eta,t^\ast \eta) \simeq Y_0\times_{\calM_G} Y_0 \ ,
\end{equation*}
where $\Isom^G$ denotes isomorphisms preserving the dual graph $G$.  
\end{proof}

\begin{proof}[Alternative proof of Corollary \ref{cor_codim=numedges}]
The Betti number $b_1(G)$ can be can calculated by $b_1(G)=\#E(G)-\#V(G)+1$ and by Proposition \ref{prop_MGstructure} we have $\dim\calM_G=\dim\widetilde{\calM}_G$. Using $g=b_1(G)+\sum_v h(v)$ as well as $\sum_vn_v=n+2\cdot\#E(G)$, we therefore obtain:
\begin{equation*}\begin{split}
\dim\calM_G=\dim\widetilde{\calM}_G&=\sum_{v\in V(G)}3h(v)-3+n_v \\
&=3\big(b_1(G)+\sum_{v\in V(G)}h(v)\big)-3\big(\#V(G)-b_1(G)\big) +\sum_{v\in V(G)}n_v\\
&=3g-3+n +2\cdot\#E(G) -3\cdot\#E(G)\\
&=\dim\calMbar_{g,\calA} -\#E(G) \ .
\end{split}\end{equation*}
\end{proof}

\section{Deformation retraction onto the non-Archimedean skeleton}

The goal of this section is to prove our main result, Theorem \ref{thm_modulartrop}. We begin with a quick review of the construction of the deformation retraction onto the skeleton of a simple toroidal scheme from \cite[Section 3.1]{Thuillier_toroidal} in Section \ref{section_skeletonsimple} and more generally of a toroidal Deligne-Mumford stack from \cite[Section 6]{AbramovichCaporasoPayne_tropicalmoduli} in Section \ref{section_skeletonstack}. Section \ref{section_proof} contains the proof of Theorem \ref{thm_modulartrop}.

\subsection{Skeletons of simple toroidal schemes}\label{section_skeletonsimple}

Suppose that $X_0\hookrightarrow X$ is a simple toroidal embedding, that is an open embedding such that for every point $x\in X$ there is an open neighborhood $U$ of $x$ and an \'etale morphism $\gamma\mathrel{\mathop:}U\rightarrow Z$ into a toric variety $Z$ with big torus $T$ such that $\gamma^{-1}(T)=X_0\cap U$. In \cite[Section 3.2]{Thuillier_toroidal} Thuillier defines a strong deformation retraction $\mathbf{p}$ from the non-Archimedean analytic space $X^\beth$ as defined in \cite[Section 1]{Thuillier_toroidal} onto a closed subset $\frakS(X)$ of $X^\beth$, the \emph{non-Archimedean skeleton} of $X$. 

We denote by $S_+$ the sheaf monoids associating to an open subset $U$ of $X$ the monoid $S_+(U)$ of effective Cartier divisors with support fully contained in $X-X_0$. As shown in \cite[Section 3.1]{Thuillier_toroidal} the natural stratification of the toric varieties $Z$ by $T$-orbits lifts to give a well-defined stratification of $X$ by locally closed subsets, henceforth called the \emph{toroidal strata} of $X$. Note that the unique open subset of this stratification is $X_0$. Denote by $F_X$ the set of generic points of the toroidal strata together with the induced topology and endowed with the restriction of the $S_+$. By \cite[Proposition 9.2]{Kato_toricsing} the monoidal space $F_X$ has the structure of what is called a Kato fan in \cite{Ulirsch_functroplogsch} and comes with a natural characteristic morphism $\phi_X\mathrel{\mathop:}(X,S_+)\rightarrow F_X$ sending every point in a toroidal stratum to its generic point. We refer the reader to \cite{Ulirsch_functroplogsch} for details on this construction.

In particular, by \cite[Theorem 1.2]{Ulirsch_functroplogsch} Thuillier's strong deformation retraction can be described as follows: 
\begin{itemize}
\item The skeleton $\frakS(X)$ is naturally homeomorphic to the set $\Sigmabar_{X}=F_X(\Rbar_{\geq 0})$ of $\Rbar_{\geq 0}=\R_{\geq 0}\sqcup\{\infty\}$-valued points.
\item A point $x$ in $X^\beth$ gives rise to a morphism $\underline{x}\mathrel{\mathop:}\Spec R\rightarrow (X,S_+)$ of monoidal spaces, where $R$ is some non-Archimedean extension of $k$, and the image of $x$ in $\Sigmabar_X$ is given by the composition
\begin{equation*}\begin{CD}
\Spec\Rbar_{\geq 0}@>\val^\#>>\Spec R@>\underline{x}>>(X,S_+)@>\phi_X>>F_X ,
\end{CD}\end{equation*}
where $\val^\#$ is the morphism induced by the valuation of $R$. 
\end{itemize}

\subsection{Skeletons of toroidal Deligne-Mumford stacks} \label{section_skeletonstack}
Suppose now that $\calX_0\hookrightarrow\calX$ is toroidal embedding of separated Deligne-Mumford stacks of finite type over $k$, i.e. an open embedding of Deligne-Mumford stacks admitting a surjective \'etale morphism $U\rightarrow\calX$ such that the base change $U_0\hookrightarrow U$ is a simple toroidal embedding. The toroidal stratification of $U$ induces a natural toroidal stratification of $\calX$ by locally closed substacks $\calE$ that does not depend on the choice of $U$. We write $S_+$ for the \'etale sheaf of effective Cartier-divisors with support in $\calX-\calX_0$.

The Keel-Mori Theorem \cite{KeelMori_groupoidquotients} implies that the stack $\calX$ has a coarse moduli space $X$, which has the structure of separated algebraic space. By \cite{ConradTemkin_algspaces} the analytification $X^{an}$ of $X$ exists in the category of analytic spaces and, following \cite[Definition 6.1.2]{AbramovichCaporasoPayne_tropicalmoduli}, we can define $X^\beth$ as the subspace of $X^{an}$ that is locally given by unit balls in $X^{an}$. Note that the valuative criterion for properness yields $X^\beth=X^{an}$, whenever $\calX$ is proper over $k$. 

In Section \cite[Section 6.1]{AbramovichCaporasoPayne_tropicalmoduli} the authors extend Thuillier's \cite{Thuillier_toroidal} construction and show that this datum defines a strong deformation retraction $\mathbf{p}$ of $X^\beth$ onto a closed subset $\mathfrak{S}(\calX)$ of $X^\beth$, which is again called the \emph{non-Archimedean skeleton} of $\calX$. 

Consider now the category $\calH_\calX$ defined as follows:
\begin{itemize}
\item Its objects are the generic points of the toroidal strata of $\calX$.
\item The morphisms in $\calH_\calX$ are generated by the natural homomorphisms $(S_+)_\eta\rightarrow (S_+)_\xi$, whenever $\eta$ specializes to $\xi$, and the \emph{monodromy groups} $H_\xi$ at $\xi$. \end{itemize} 

Recall that the sheaf $S_+$ is \'etale locally trivial on the toroidal strata of $\calX$ by \cite[Proposition 6.2.1]{AbramovichCaporasoPayne_tropicalmoduli}. The \emph{monodromy group} $H_\xi$ consists of those automorphisms of $(S_+)_\eta$ that are induced by the operation of $\pi_1^{et}(\calE_\xi,\xi)$ on $(S_+)_\xi$, where $\calE_\xi$ is the unique toroidal stratum containing $\xi$.

There is a natural functor $\Sigma\mathrel{\mathop:}\calH_\calX\rightarrow \mathbf{RPCC}$ given by 
\begin{itemize}
\item the association $\xi\rightarrow\sigma_\xi=\Hom\big((S_+)_\xi,\R_{\geq 0}\big)$,
\item the embedding of a face $\sigma_\eta\hookrightarrow\sigma_\xi$, whenever $\eta$ specializes to $\xi$, and
\item an automorphism of $\sigma_\xi$ for every automorphism of the monodromy group $H_\xi$.
\end{itemize}

This functor naturally extends to a functor $\Sigmabar$ into the category of extended rational polyhedral complexes, given by $\xi\rightarrow \sigma_\xi=\Hom\big((S_+)_\xi,\Rbar_{\geq 0}\big)$.

We can now rephrase \cite[Proposition 6.2.6]{AbramovichCaporasoPayne_tropicalmoduli} as follows:

\begin{proposition}[ \cite{AbramovichCaporasoPayne_tropicalmoduli} Proposition 6.2.6]\label{prop_skeleton=colimit} The skeleton $\mathfrak{S}(\calX)$ is the colimit 
\begin{equation*}
\mathfrak{S}(\calX)=\lim_{\longrightarrow} \sigmabar_\xi \ ,
\end{equation*}
taken over the category $\calH_\calX$. Let $x\in X^\beth$ be represented by a morphism $\underline{x}\mathrel{\mathop:}\Spec R\rightarrow \calX$ from a valuation ring extending $k$ and write $p_x$ for the image of the closed point in $\calX$. Then $(S_+)_{p_x}=(S_+)_{\xi_x}$ for the generic point $\xi_x$ of the unique stratum containing $p_x$ and the image $\mathbf{p}(x)$ in $\sigmabar_{\xi_x}=\Hom\big((S_+)_{\xi_x},\Rbar_{\geq 0}\big)$ is given by the composition 
\begin{equation*}\begin{CD}
\Spec \Rbar_{\geq 0}@>\val^\#>>\Spec R @>>>\Spec \widehat{\calO}_{\calX,p_x}@>>>\Spec (S_+)_{p_x} \ .
\end{CD}\end{equation*}
\end{proposition}

\begin{remark}
Suppose that $\calX$ is a proper toroidal Deligne-Mumford stack of finite type over $k$. By \cite[Section 1.5]{Ulirsch_nonArchstacks} the skeleton $\frakS(\calX)$ is actually a deformation retract of the underlying topological space $\vert \calX^{an}\vert$ of the analytic stack $\calX^{an}$ associated to $\calX$. 
\end{remark}

\subsection{The skeleton of $\overline{\mathcal{M}}_{g,\mathcal{A}}$}\label{section_proof}
By Theorem \ref{thm_SNC} the open embedding $\calM_{g,\calA}\hookrightarrow\calMbar_{g,\calA}$ defines  a toroidal structure on $\calMbar_{g,\calA}$. Note the toroidal stratification is exactly the stratification of $\calMbar_{g,\calA}$ by dual graphs introduced in Section \ref{section_stratification}. Denote by $\xi_G$ the generic point of the stratum $\calM_G$. The following Lemma \ref{lemma_pi1=Aut} is a generalization of \cite[Proposition 7.2.1]{AbramovichCaporasoPayne_tropicalmoduli}.

\begin{lemma}\label{lemma_pi1=Aut}
The association $G\mapsto \xi_G$ defines a natural equivalence between the categories $(\calG_{g,\calA})^{op}$ and $\calH_{\calMbar_{g,\calA}}$.
\end{lemma}

\begin{proof}
Note first that weighted graph contraction $G'\rightarrow G$ are in an unique order-reversing one-to-one correspondence with the specialization relations $\xi_G\rightarrow\xi_{G'}$. Therefore it is enough to show that the image of $\pi_1^{et}(\calMbar_{g,\calA},\xi_G)$ acting on the the $(S_+)_{\xi_G}$ is precisely $\Aut(G)$. Consider the Galois cover $\widetilde{\calM}_G\rightarrow\calM_G$. The operation of $\pi_1^{et}(\calMbar_{g,\calA},\xi_G)$ on the sheaf pullback of $S_+$ to $\widetilde{M}_G$ is trivial and it therefore factors through its quotient $\Aut(G)$.
\end{proof}

\begin{proof}[Proof of Theorem \ref{thm_modulartrop}]
Recall that $\Mbar_{g,\calA}^{trop}$ is defined as the colimit
\begin{equation*}
\Mbar_{g,\calA}^{trop}=\lim_{\longrightarrow} \sigmabar_G
\end{equation*}
taken over the category $(\calG_{g,\calA})^{op}$ and that by Proposition \ref{prop_skeleton=colimit} the skeleton $\frakS(\calMbar_{g,\calA})$ is given as the colimit
\begin{equation*}
\mathfrak{S}(\calMbar_{g,\calA})=\lim_{\longrightarrow} \sigmabar_\xi 
\end{equation*}
taken over the category $\calH_{\calMbar_{g,\calA}}$. Therefore Lemma \ref{lemma_pi1=Aut} immediately implies that there is an isomorphism
\begin{equation*}
J_{g,\calA}\mathrel{\mathop:}\Mbar_{g,\calA}^{trop}\longrightarrow\mathfrak{S}(\calMbar_{g,\calA})  \ .
\end{equation*}
of generalized extended cone complexes.

We finally show that the strong deformation retraction $\mathbf{p}\mathrel{\mathop:}\Mbar_{g,\calA}^{an}\rightarrow\mathfrak{S}(\calMbar_{g,\calA})$ can be given a modular interpretation as stated in the introduction. By the valuative criterion of properness a point $x\in \Mbar_{g,\calA}^{an}$ can be represented by a morphism $\Spec R\rightarrow\calMbar_{g,\calA}$, which, in turn, gives rise to a $(g,\calA)$-stable curve $\calC_x\rightarrow \Spec R$ over $R$. Denote the dual graph of its special fiber $\calC_s$ by $G_x$ and the image of the closed point in $\Spec R$ in $\calMbar_{g,\calA}$ by $p_x$. 

By Theorem \ref{thm_nodecoordinates} we can choose coordinates $t_1,\ldots, t_N$ in $\widehat{\calO}_{\calMbar_{g,\calA},p_x}$ such that the locus, where $\calC_s$ remains singular is given by $t_1\ldots t_k=0$. In formal coordinates we can describe $\calC$ around a node $q_i$ of $\calC_s$ by $xy=f_i$, where the $f_i\in R$ are precisely the images of $t_i$ in $R$. Now both the deformation retraction $\mathbf{p}$ and $\trop_{g,\calA}$ are given by associating to $x$ the element in $\mathfrak{S}(\calMbar_{g,\calA})=\Mbar_{g,\calA}^{trop}$ represented by $\big(\val(f_1),\ldots,\val(f_k),0,\ldots, 0\big)$ in $\sigmabar_{G_x}$. This shows $\mathbf{p}(x)=\trop_{g,\calA}(x)$ and finishes the proof of Theorem \ref{thm_modulartrop}.
\end{proof}

\section{Tropical tautological maps}\label{section_troptaut}

The purpose of this section is to define tropical analogues of the tautological maps between the moduli spaces $\calMbar_{g,\calA}$ generalizing the constructions in \cite[Section 8]{AbramovichCaporasoPayne_tropicalmoduli}. We require the tropical tautological maps to commute with the tropicalization map
\begin{equation*}
\trop_{g,\calA}\mathrel{\mathop:}\Mbar_{g,\calA}^{an}\longrightarrow\Mbar_{g,\calA}^{trop}\end{equation*}
as a basic principle to justify that our definitions make sense.

\subsection{Forgetful and reduction morphisms}\label{section_forgetful&reduction}

Fix an input datum $(\calA,g)$ and let $\calB=(b_1,\ldots,b_n)$ be another tuple of weights such that $b_i\leq a_i$ for all $1\leq i\leq n$. In \cite[Theorem 4.1]{Hassett_weightedmoduli} Hassett constructs a natural birational \emph{reduction morphism} 
\begin{equation*}
\rho_{\calA,\calB}\mathrel{\mathop:}\calMbar_{g,\calA}\longrightarrow\calMbar_{g,\calB}
\end{equation*}
that takes an element $(C,p_1,\ldots, p_n)$ of $\calMbar_{g,\calA}$ and collapses all the components along which the divisor $K_C=b_1p_1+\ldots+b_np_n$ fails to be ample. 

Moreover consider a subset $\calA'=\{a_{i_1},\dots,a_{i_r}\}\subseteq\calA$ such that $2g-2+a_{i_1}+\ldots +a_{i_r}>0$. By \cite[Theorem 4.3]{Hassett_weightedmoduli} there is a natural \emph{forgetful morphism}
\begin{equation*}
\phi_{\calA,\calA'}\mathrel{\mathop:}\calMbar_{g,\calA}\longrightarrow\calMbar_{g,\calA'} 
\end{equation*}
that can be described by associating to an $\calA$-stable curve $(C,p_1,\ldots, p_n)$ in $\calMbar_{g,\calA}$ the curve $\phi_{\calA,\calA'}(C,p_1,\ldots, p_n)$ given by deleting the marked points $p_i$ with $i\notin\calA'$ and successively collapsing the components of $C$ such that $K_C+a_{i_1}p_{i_1}+\ldots +a_{i_r}p_{i_r}$ is not ample.

\begin{proposition}\label{prop_tropicalreduction&forgetful}
There is a natural \emph{tropical reduction map}
\begin{equation*}
\rho_{\calA,\calB}^{trop}\mathrel{\mathop:}\Mbar_{g,\calA}^{trop}\longrightarrow\Mbar_{g,\calB}^{trop}
\end{equation*} 
and a natural \emph{tropical forgetful map} 
\begin{equation*}
\phi_{\calA,\calA'}^{trop}\mathrel{\mathop:}\Mbar_{g,\calA}^{trop}\longrightarrow\Mbar_{g,\calA'}^{trop}
\end{equation*}
making the diagrams
\begin{equation*}
\begin{CD}
\Mbar_{g,\calA}^{an}@>\trop_{g,\calA}>>\Mbar_{g,\calA}^{trop}\\
@V\rho_{\calA,\calB}^{an}VV@VV\rho_{\calA,\calB}^{trop}V\\
\Mbar_{g,\calB}^{an}@>\trop_{g,\calB}>>\Mbar_{g,\calB}^{trop}
\end{CD}
\qquad\qquad
\begin{CD}
\Mbar_{g,\calA}^{an}@>\trop_{g,\calA}>>\Mbar_{g,\calA}^{trop}\\
@V\phi_{\calA,\calA'}^{an}VV@VV\phi_{\calA,\calA'}^{trop}V\\
\Mbar_{g,\calA'}^{an}@>\trop_{g,\calA'}>>\Mbar_{g,\calA'}^{trop}
\end{CD}
\end{equation*}
commutative.
\end{proposition}

Proposition \ref{prop_tropicalreduction&forgetful} is an immediate consequence of Hassett's description of the forgetful and reduction morphisms for $\calMbar_{g,\calA}$ in \cite[Section 4.1]{Hassett_weightedmoduli} as well as the reasoning in \cite[Section 8.2]{AbramovichCaporasoPayne_tropicalmoduli}. We provide a proof in our language for the convenience of the reader. 

\begin{proof}[Proof of Propostion \ref{prop_tropicalreduction&forgetful}]
We shall prove both statements simultaneously using the notation $\psi_{\calA,\calB}$ for both the reduction morphism and the forgetful morphism. To make this notation consistent we follow \cite[Section 4.1]{Hassett_weightedmoduli} and formally set $\calB=\calA'\cup\{0,\ldots,0\}$ as well as:
\begin{itemize}
\item $\calMbar_{g,\calB}=\calMbar_{g,\calA'}$
\item $\Mbar_{g,\calB}^{trop}=\Mbar_{g,\calA'}^{trop}$
\item $\calG_{g,\calB}=\calG_{g,\calA'}$
\item $\Sigmabar_{g,\calB}=\Sigmabar_{g,\calA'}$
\end{itemize}
 
Our approach is to define natural functors 
\begin{equation*}
\psi_{\calA,\calB}^{\calG}\mathrel{\mathop:}\calG_{g,\calA}\rightarrow\calG_{g,\calB}
\end{equation*}
and 
\begin{equation*}
 \psi_{\calA,\calB}^{\Sigmabar}\mathrel{\mathop:}\Sigmabar_{g,\calA}\rightarrow\Sigmabar_{g,\calB}\end{equation*}
that will induce $\psi_{\calA,\calB}^{trop}$ by the universal property of colimits.

Let $G$ be a $\calA$-stable weighted graph. If $G$ is not $\calB$-stable we find ourselves in one of the following two situations:
\begin{enumerate}[(i)]
\item There is a vertex $v\in V(G^\ast)$ such that $h(v)=0$, $\vert v\vert_{E}=1$, and 
\begin{equation*}
2h(v)-2+\vert v\vert_E+\vert v\vert_{\calB}=\vert v\vert_{\calB} -1\leq 0 \ . 
\end{equation*}
In this case we contract the unique edge $e$ incident to $v$ and attach all the legs of $G^\ast$ incident to $v$ to the vertex on the other end of $e$. 
\item There is a vertex $v\in V(G^\ast)$ such that $h(v)=0$, $\vert v\vert_E=2$, and 
\begin{equation*}
2h(v)-2+\vert v\vert_E+\vert v\vert_{\calB}=\vert v\vert_{\calB} = 0 \ . 
\end{equation*}
In this case the graph $G$ does not have any legs of positive weight incident to $v$ and we replace the two adjacent edges $e_1$ and $e_2$ by one edge connecting the two vertices $v_1$ and $v_2$ on the other end of $v$.
\end{enumerate}

Applying the algorithm described in (i) and (ii) possibly multiple times and deleting legs of zero weight we obtain a functor  
\begin{equation*}\begin{split}
\psi_{\calA,\calB}^\calG\mathrel{\mathop:}\calG_{g,\calA}&\longrightarrow\calG_{g,\calB} \\
G&\longmapsto G^\ast \ ,
\end{split}\end{equation*}
since automorphisms of $G$ induce automorphisms of $G^\ast$ and weighted edge contractions of $G$ naturally induced weighted edge contractions of $G^\ast$.
Moreover, the projection maps $\sigmabar_G\rightarrow\sigmabar_{G^\ast}$ induce a functor $\psi_{\calA,\calB}^{\Sigmabar}\mathrel{\mathop:}\Sigmabar_{g,\calA} \rightarrow\Sigmabar_{g,\calA}$ making the diagram
\begin{equation*}\begin{CD}
\calG_{g,\calA}@>\Sigmabar_{g,\calA}>>\Sigmabar_{g,\calA}\\
@V\psi_{\calA,\calA'}^\calG VV @VV\psi_{\calA,\calA'}^{\Sigmabar} V\\
\calG_{g,\calA'}@>\Sigmabar_{g,\calA'}>>\Sigmabar_{g,\calA'}
\end{CD}\end{equation*}
commutative.

The map $\psi_{\calA,\calB}^{trop}$ is defined to be the map $\Mbar_{g,\calA}^{trop}\rightarrow\Mbar_{g,\calB}^{trop}$ induced by $\psi_{\calA,\calB}^{\calG}$ and $\psi_{\calA,\calB}^{\Sigmabar}$ using the universal property of colimits.

Now note that the morphism $\psi_{\calA,\calB}$ induces a functor $\psi_{\calA,\calB}^\calH\mathrel{\mathop:}\calH_{\calMbar_{g,\calA}}\rightarrow\calH_{\calMbar_{g,\calB}}$ making the natural diagram
\begin{equation*}\begin{CD}
\calH_{\calMbar_{g,\calA}} @<\simeq<< (\calG_{g,\calA})^{op}\\
@V\psi_{\calA,\calA'}^\calH VV @VV\psi_{\calA,\calA'}^\calG V\\
\calH_{\calMbar_{g,\calB}} @<\simeq<< (\calG_{g,\calB})^{op}
\end{CD}\end{equation*}
commutative. Because of that and Theorem \ref{thm_modulartrop} the diagrams in the statement of Theorem \ref{prop_tropicalreduction&forgetful} commute.
\end{proof}

\begin{proposition}\label{prop_reduction=subcomplex}
The tropical reduction morphism $\phi_{\calA,\calB}^{trop}$ has a section identifying the moduli space $\Mbar_{g,\calB}^{trop}$ with the subcomplex of $\Mbar_{g,\calA}^{trop}$ given by removing those extended relatively open cones $\sigmabar^\circ_G$ such that $G$ is not $\calB$-stable.
\end{proposition}

\begin{proof}
Suppose that $G$ is weighted graph that is stable of type $(g,\calA)$ but not of type $(g,\calB)$. Then all other weighted graphs $G'$ that can be contracted to $G$ are also not stable of type $(g,\calB)$. On the other hand all graphs $G$ that are stable of type $(g,\calB)$ are also stable of type $(g,\calA)$ and their reduction $\rho_{\calA,\calB}(G)$ is equal to $G$ itself. 
\end{proof}

\begin{remark}
Similar sections exists for the forgetful morphism in both the algebraic and the tropical world (see \cite[Proposition 8.2.4]{AbramovichCaporasoPayne_tropicalmoduli} for the case $\calA=(1,\ldots,1)$). Unlike this case, the section of the reduction morphism constructed in Proposition \ref{prop_reduction=subcomplex} does not have an algebraic analogue. One can, however, define a continuous section of $\phi_{\calA,\calB}$ on the level of underlying topological spaces. 
\end{remark}

\subsection{Clutching and gluing}\label{section_tropicalclutching&gluing}

Let $(g,\calA)$ be a fixed input datum and $G$ be a weighted graph that is stable of type $(g,\calA)$. Recall from Section \ref{section_algebraicclutching&gluing} that for a vertex $v$ of $G$ we denote by $\calA(v)$ the tuple 
\begin{equation*}
(a_{i_1},\ldots a_{i_k},1,\ldots,1)
\end{equation*}
consisting of those $a_i$ that correspond to legs $l_i$ emanating from $v$ and a $1$ for every flag of an edge incident to $v$.

\begin{definition}\label{def_tropicalclutching&gluing}
In analogy with the algebraic situation in Section \ref{section_algebraicclutching&gluing} we define the \emph{tropical clutching and gluing map} 
\begin{equation*}
\phi_G^{trop}\mathrel{\mathop:}\prod_{v\in V(G)}\Mbar_{h(v),\calA(v)}^{trop}\longrightarrow \Mbar_{g,\calA}^{trop}
\end{equation*}
given by connecting two legs of tropical curves $[\Gamma_v]\in \Mbar_{h(v),\calA(v)}^{trop}$ by a bridge at infinity whenever the corresponding flags in $G$ are connected by an edge.
\end{definition}

\begin{proposition}\label{prop_tropicalclutching&gluing}
The natural diagram
\begin{equation*}\begin{CD} 
\prod_{v\in V(G)}\Mbar_{h(v),\calA(v)}^{an}@>\prod\trop_{h(v),\calA(v)}>> \prod_{v\in V(G)}\Mbar_{h(v),\calA(v)}^{trop}\\
@V\phi_G^{an}VV @VV\phi_G^{trop}V\\
\Mbar_{g,\calA}^{an}@>\trop_{g,\calA}>>\Mbar_{g,\calA}^{trop}
\end{CD}\end{equation*}
is commutative
\end{proposition}

\begin{proof}
Let $(\calC^v,p_1^v,\ldots,p_{n_v}^v)$ be families of $\big(h(v),\calA(v)\big)$ stable curves over a valuation ring $R$ extending $k$ and denote the tropical curves associated to this data by $\Gamma_v$. Observe that the clutching and gluing map $\phi_G$ applied to the $(\calC^v,p_1^v,\ldots,p_{n_v}^v)$ exactly corresponds to connecting two legs of the $\Gamma_v$, whenever they correspond to an edge in $G$. Since $\phi_G\big((\calC^v,p_1^v,\ldots,p_{n_v}^v)\big)$ has a node over all of $\Spec R$, whenever two marked points have been glued, the special fiber is given by $xy=0$ in formal coordinates and therefore the connecting edge in $\phi_G^{trop}\big((\Gamma_v)\big)$ has to be of infinite length. 
\end{proof}

As special cases of Definition \ref{def_tropicalclutching&gluing} we obtain the following two maps:
\begin{itemize}
\item The \emph{tropical clutching map} 
\begin{equation*}
\kappa^{trop}=\kappa_{g_1,\calA_1,g_2,\calA_2}^{trop}\mathrel{\mathop:} \Mbar_{g_1,\calA_1\cup\{1\}}^{trop}\times\Mbar_{g_2,\calA_2\cup\{1\}}^{trop}\longrightarrow \Mbar_{g,\calA}^{trop}
\end{equation*}
is given by sending a pair of extended tropical curves $\Gamma_1$ and $\Gamma_2$ in $\Mbar_{g_1,\calA_1\cup\{1\}}^{trop}$ and $\Mbar_{g_2,\calA_2\cup\{1\}}^{trop}$ respectively to the extended tropical curve $\Gamma$ that is obtained by connecting the two legs $l_{n_1+1}$ and $l_{n_2+1}$ at infinity. 
\item The \emph{tropical gluing map}
\begin{equation*}
\gamma^{trop}\mathrel{\mathop:}\Mbar^{trop}_{g-1,\calA\cup\{1,1\}}\longrightarrow\Mbar^{trop}_{g,\calA}
\end{equation*}
is defined by sending an extended tropical curve $\Gamma$ in $\Mbar_{g-1,\calA\cup\{1,1\}}$ to the tropical curve $\widetilde{\Gamma}$ obtained by connecting the two legs $l_{n+1}$ and $l_{n+2}$ at infinity. 
\end{itemize}
Using this notation Proposition \ref{prop_tropicalclutching&gluing} yields the following generalization of \cite[Theorem 1.2.2]{AbramovichCaporasoPayne_tropicalmoduli}:

\begin{corollary}\label{cor_tropicalclutching&gluing}
The natural diagrams
\begin{equation*}\begin{CD}
 \Mbar_{g_1,\calA_1\cup\{1\}}^{an}\times\Mbar_{g_2,\calA_2\cup\{1\}}^{an}@>\trop_{g_1,\calA_1\cup\{1\}}\times\trop_{g_2,\calA_2\cup\{1\}}>>\Mbar_{g_1,\calA_1\cup\{1\}}^{trop}\times\Mbar_{g_2,\calA_2\cup\{1\}}^{trop}\\
 @V\kappa^{an}VV @VV k^{trop}V\\
  \Mbar_{g,\calA}^{an} @>\trop_{g,\calA}>> \Mbar_{g,\calA}^{trop}
\end{CD}\end{equation*}
and
\begin{equation*}\begin{CD}
\Mbar_{g-1,\calA\cup\{1,1\}}^{an}@>\trop_{g-1,\calA\cup\{1,1\}}>>\Mbar_{g-1,\calA\cup\{1,1\}}^{trop}\\
@V\gamma^{an}VV @VV\gamma^{trop}V\\
\Mbar_{g,\calA}^{an}@>\trop_{g,\calA}>>\Mbar_{g,\calA}^{trop}
\end{CD}\end{equation*}
are commutative.
\end{corollary}

\section{Variations of weight data}\label{section_wallsandchambers}

\subsection{Chamber decompositions} 

In this section we compare how, given a fixed genus $g$, the two functions $\calA\mapsto\calMbar_{g,\calA}$ and $\calA\mapsto\Mbar_{g,\calA}^{trop}$ vary in $\calA$. Fix a genus $g\geq 0$ and a number $n\geq 0$. We denote by $\scrD_{g,n}$ the set of possible weight data 
\begin{equation*}
\scrD_{g,n}=\big\{(a_1,\ldots,a_n)\in(0,1]^n\cap\QQ^n\big\vert a_1+\ldots + a_n>2-2g \big\} \ .
\end{equation*}
As in \cite[Section 5]{Hassett_weightedmoduli} a \emph{chamber decomposition} $\scrW$ of $\scrD_{g,n}$ consists of a finite set of hyperplanes $w_S\subseteq\scrD_{g,n}$. We refer to the $w_S$ as the \emph{walls} of the chamber decomposition $\scrW$ and to connected components of the complement of the $w_S$ as the \emph{open chambers} of $\scrW$.

In \cite[Section 5]{Hassett_weightedmoduli} Hassett studies chamber decompositions of $\scrD_{g,n}$ with the property that the functions $\calA\mapsto\calMbar_{g,\calA}$ and $\calA\mapsto\calC_{g,\calA}$ are constant on every open chamber, where $\calC_{g,\calA}$ denotes the universal curve of $\calMbar_{g,\calA}$. 

\begin{proposition}\label{prop_algcham>=tropcham}
Suppose that $\scrW$ is a chamber decomposition of $\scrD_{g,n}$ such that $\calA\mapsto\calC_{g,\calA}$ is constant on open chambers. Then $\calA\mapsto\Mbar_{g,\calA}^{trop}$ is constant on the open chambers of $\scrW$ as well.
\end{proposition}

\begin{proof}
If the two universal curves $\calC_{g,\calA}$ and $\calC_{g',\calA'}$ are isomorphic, there is an isomorphism between the moduli stacks $\calMbar_{g,\calA}$ and $\calMbar_{g',\calA'}$ that preserves the stratifications by dual graphs.  Using Theorem \ref{thm_modulartrop} we see that the tropical moduli spaces $\Mbar_{g,\calA}^{trop}$ and $\Mbar_{g',\calA'}^{trop}$ are isomorphic.
\end{proof}

Moreover Hassett considers the \emph{coarse chamber decomposition}, which is given by 
\begin{equation*}
\scrW_c=\Big\{\sum_{j\in S}a_j=1\big\vert S\subseteq\{1,\ldots, n\} \textrm{ and } 2<\vert S\vert \leq n-2\delta_{g,0}\Big\} \ ,
\end{equation*}
as well as the \emph{fine chamber decomposition}, which is given by 
\begin{equation*}
\scrW_f=\Big\{\sum_{j\in S}a_j=1\big\vert S\subseteq\{1,\ldots, n\} \textrm{ and } 2\leq\vert S\vert \leq n-2\delta_{g,0}\Big\} \ .
\end{equation*}
Hereby $\delta_{i,j}$ denotes the Kronecker delta 
\begin{equation*}
\delta_{i,j}=\left\{
  \begin{array}{l l}
   1 & \quad \text{if } i=j\\
   0 & \quad \text{else.}\\
  \end{array} \right .
\end{equation*}

In \cite[Proposition 5.1]{Hassett_weightedmoduli} Hassett shows that $\scrW_c$ is the coarsest chamber decompostion of $\scrD_{g,n}$ such that $\calA\mapsto\calMbar_{g,\calA}$ is constant on every open chamber and $\scrW_f$ is the coarsest chamber decomposition of $\scrD_{g,n}$ such that the map $\calA\mapsto\calC_{g,\calA}$ is constant. Therefore Proposition \ref{prop_algcham>=tropcham} immediately implies that the association $\calA\mapsto \Mbar_{g,\calA}^{trop}$ is constant on the fine chambers of $\scrD_{g,n}$. 

The following Proposition \ref{prop_algfinecham=tropcoarsecham} is a partial analogue of \cite[Proposition 5.1]{Hassett_weightedmoduli} in the tropical world. 

\begin{proposition}\label{prop_algfinecham=tropcoarsecham}
The fine chamber decomposition $\scrW_f$ is the coarsest chamber decomposition of $\scrD_{g,n}$ such that $\calA\mapsto\Mbar_{g,\calA}^{trop}$ is constant on every open chamber.
\end{proposition}

\begin{proof}
The map $\calA\mapsto \Mbar_{g,\calA}^{trop}$ is constant on the fine chambers of $\scrD_{g,n}$. So it is enough to note that $\Mbar_{g,\calA}^{trop}$ changes whenever we cross a wall of the fine chamber decomposition $\scrW_f$.

Suppose first that $g\geq 1$. Let $S\subseteq\{1,\ldots, n\}$ with $2\leq\vert S\vert\leq n$. Consider the graph $G_S$ containing one edge between two vertices $v_0$ and $v_g$, one of weight $0$ and the other of weight $g$, and legs $l_i$ incident to $v_0$, whenever $i\in S$ and incident to $v_g$, whenever $i\notin S$.
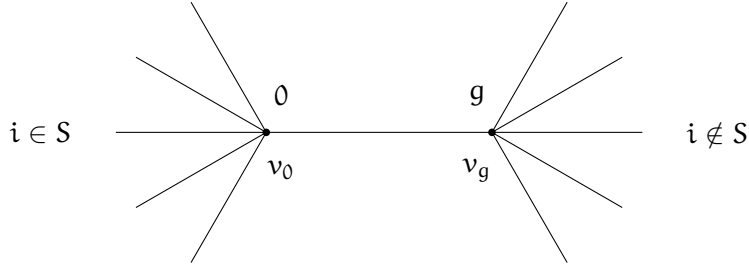
\begin{center}\begin{tikzpicture}
\fill (0,0) circle (0.05 cm);
\fill (3,0) circle (0.05 cm);
      
\draw (0,0) -- (3,0);

\foreach \a in {120,150,180,210,240}
\draw (\a:0) -- (\a:2);

\foreach \a in {300,330,0,30,60}
\draw (3,0) -- ({3+2*cos(\a)},{2*sin(\a)});

\node at (0.2,0.5) {$0$};
\node at (0.2,-0.5) {$v_0$};
\node at (2.8,0.5) {$g$};
\node at (2.8,-0.5) {$v_g$};
\node at (-3,0) {$i\in S$};
\node at (6,0) {$i\notin S$};

\end{tikzpicture}\end{center}
If $\sum_{i\in S}a_i>1$, then $G_S$ is stable of type $(g,\calA)$, since $\vert v_0\vert=1+\sum_{i\in S}a_i>2$ and $h(v_g)\geq 1$. But if $\sum_{i\in S}a_i\leq 1$, the graph $G_S$ is not stable of type $(g,\calA)$. 

Consider now the case $g=0$. Let $S\subseteq\{1,\ldots,n\}$ with $2\leq\vert S\vert\leq n-2$ and consider again the same graph $G_S$ as above with two vertices of weight $0$ connected by an edge and legs incident to $v_i$ depending on whether they are in $S$ or not. 
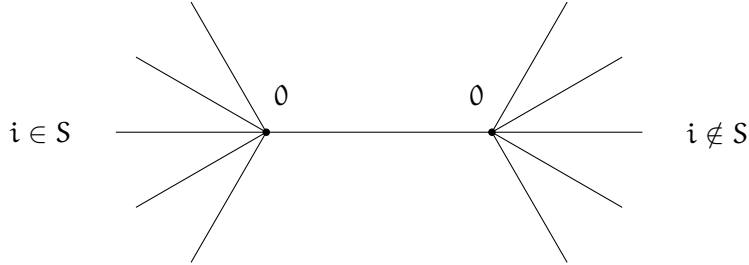
\begin{center}\begin{tikzpicture}
\fill (0,0) circle (0.05 cm);
\fill (3,0) circle (0.05 cm);
      
\draw (0,0) -- (3,0);

\foreach \a in {120,150,180,210,240}
\draw (\a:0) -- (\a:2);

\foreach \a in {300,330,0,30,60}
\draw (3,0) -- ({3+2*cos(\a)},{2*sin(\a)});

\node at (0.2,0.5) {$0$};
\node at (2.8,0.5) {$0$};
\node at (-3,0) {$i\in S$};
\node at (6,0) {$i\notin S$};

\end{tikzpicture}\end{center}
Suppose that $\sum_{i\in S}a_i\leq 1$. Then $\sum_{i=1}^{n} a_i>2$ implies $\sum_{i\notin S}a_i>1$. So when crossing the wall $\sum_{i\in S}a_i=1$ without changing the $a_i$ with $i\notin S$ we obtain that $G_S$ is stable of type $(0,\calA)$, if $\sum_{i\in S}a_i>1$, and not, if $\sum_{i\in S}a_i\leq 1$. 
\end{proof}

\begin{remarks}\begin{enumerate}[(i)]
\item As noted in \cite[Remark 2.3]{AlexeevGuy_weightedstablemaps} there is a typographical error in the definitions of coarse and fine chamber decompositions in \cite[Section 5]{Hassett_weightedmoduli}. We fix this typo following the notation of \cite[Section 0.4]{BayerManin_weightedstablemaps}.
\item In \cite[Definition 2.8]{AlexeevGuy_weightedstablemaps} Alexeev and Guy propose an alternative to chamber decompositions: They associate a simplicial complex $\Delta_\calA$ to the weights $\calA$ that has a $\vert S\vert$-dimensional simplex for every subset $S\subseteq\{1,\ldots, n\}$ with $\sum_{i\in S}a_i\leq 1$. As seen in \cite[Section 4]{AlexeevGuy_weightedstablemaps} crossing a single wall from $\calA$ to $\calB$ with $\calB\geq \calA$ in $\scrW_{g,n}$ corresponds to adding a simplex to $\Delta_\calA$ in order to obtain $\Delta_\calB$.
\end{enumerate}\end{remarks}

\subsection{Kapranov's construction of $\Mbar_{0,n}$}

By Theorem \ref{prop_tropicalreduction&forgetful} we can realize $\Mbar_{g,n}$ as a composition of tropical reduction maps starting at a $\Mbar_{g,\calA}$ with a low weights $\calA$ and Proposition \ref{prop_reduction=subcomplex} ensures that while we are crossing a wall in $\scrW_f$ from lower weight data to bigger ones, we are only adding additional extended cones. 

In \cite[Section 6.1]{Hassett_weightedmoduli} Hassett identifies Kapranov's classical blow-up construction of $\Mbar_{0,n}$ (see \cite{Kapranov_ChowquotientsI} and \cite[Section 4.3]{Kapranov_Veronese&M0nbar}) with a sequence of reduction maps. The weights of this sequence are given by 
\begin{equation*}
\calA_{r,s}[n]=\Bigg(\underbrace{\frac{1}{n-r-1},\ldots,\frac{1}{n-r-1}}_{n-r-1 \textrm{ times}},\frac{s}{n-r-1},\underbrace{1,\ldots,1}_{r \textrm{ times}}\Bigg)
\end{equation*}
for $r=1,\ldots, n-3$ and $s=1,\ldots, n-r-2$. The sequence starts with $\Mbar_{0,\calA_{1,1}[n]}\simeq\PP^{n-3}$, at the $r$-th step the sequence of reduction maps is given by 
\begin{equation*}\begin{CD}
\Mbar_{0,\calA_{r,n-r-2}[n]}@>>> \ldots @>>>\Mbar_{0,\calA_{r,2}[n]}@>>>\Mbar_{0,\calA_{r,1}[n]} \ ,
\end{CD}\end{equation*}
and Kapranov has shown that at the last step $\Mbar_{0,\calA_{n-3,1}[n]}$ is isomorphic to $\Mbar_{0,n}$.

\begin{example}
The final weights in Hassett's interpretation of Kapranov's construction are given by
\begin{equation*}
\calA_{n-3,1}[n]=\Bigg(\frac{1}{2},\frac{1}{2},\frac{1}{2},\underbrace{1,\ldots,1}_{n-3 \textrm{ times}}\Bigg) \ .
\end{equation*} 
Let $n\geq 5$. As seen in Figure \ref{figure_counterexKapranov}, there is a rational weighted graph that is stable of type  $\big(0,(1,\ldots,1)\big)$ but not of type $\big(0,\calA_{n-3,1}[n]\big)$. Therefore the tropical moduli space $\Mbar_{0,n}^{trop}$ contains an extended cone corresponding to this graph, but $\Mbar_{0,\calA_{n-3,1}}^{trop}$ does not and thus these two spaces cannot be isomorphic. 
\end{example}

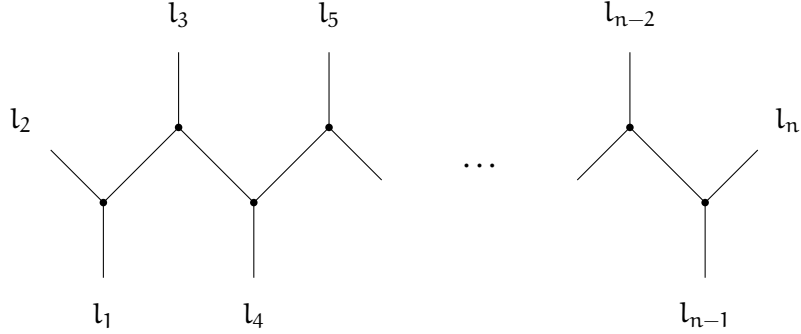
\begin{figure}

\begin{center}\begin{tikzpicture}
\fill (0,1) circle (0.05 cm);
\fill (-1,0) circle (0.05 cm);
\fill (1,0) circle (0.05 cm);
\fill (2,1,0) circle (0.05 cm);
\fill (7,0) circle (0.05 cm);
\fill (6,1) circle (0.05 cm);
      
\draw (-1,0) -- (0,1);
\draw (1,0) -- (0,1);
\draw (0,1) -- (0,2);
\draw (1,0) -- (1,-1);
\draw (-1,0) -- (-1,-1);
\draw (1,0) -- (2,1);
\draw (-1,0) -- (-1.7,0.7);
\draw (2,1) -- (2,2);
\draw (2,1) -- (2.7,0.3);
\draw (5.3,0.3) -- (6,1);
\draw (6,1) -- (6,2);
\draw (6,1) -- (7,0);
\draw (7,0) -- (7,-1);
\draw (7,0) -- (7.7,0.7);

\node at (0,2.5) {$l_3$};
\node at (2,2.5) {$l_5$};
\node at (1,-1.5) {$l_4$};
\node at (-1,-1.5) {$l_1$};
\node at (-2.1,1.1) {$l_2$};
\node at (6,2.5) {$l_{n-2}$};
\node at (7,-1.5) {$l_{n-1}$};
\node at (8.1,1.1) {$l_n$};

\node at (4,0.5) {$\ldots$};

\end{tikzpicture}\end{center}
\caption{A graph with $0$ vertex weights that is stable of type $\big(0,(1,\ldots,1)\big)$ but not of type $\big(0,\calA_{n-3,1}[n]\big)$, whenever $n\geq 5$.}\label{figure_counterexKapranov} \end{figure}

The explanation for this behavior is that, although we have an isomorphism 
\begin{equation*}
\Mbar_{0,\calA_{n-3,1}[n]}\simeq\Mbar_{0,n} \ ,
\end{equation*}
the universal curves $\calC_{0,\calA_{n-3,1}[n]}$ and $\calC_{0,n}$ of these two moduli spaces are not isomorphic. In other words $M_{0,\calA_{n-3,1}[n]}$, the locus parametrizing smooth curves in $\Mbar_{0,\calA_{n-3,1}[n]}$, is bigger than $M_{0,n}$. One can, of course, deal with this phenomenon by further increasing the weights and, by the following example, a minimal increase will be enough. 

\begin{example}
Let us now consider the weights
\begin{equation*}
\calA_\epsilon=\Bigg(\frac{1}{2}+\epsilon,\frac{1}{2}+\epsilon,\frac{1}{2}+\epsilon,\underbrace{1,\ldots,1}_{n-3 \textrm{ times}}\Bigg) \ .
\end{equation*}
for $0<\epsilon\leq\frac{1}{2}$. Every rational stable tropical curve is also stable of type $\calA_\epsilon$ and the tropical reduction map therefore induces a natural isomorphism
\begin{equation*}
\Mbar_{0,n}^{trop}\simeq\Mbar_{0,\calA_\epsilon}^{trop} \ .
\end{equation*}
This, together with the fact that for every rational $\calA_\epsilon$-stable $n$-marked curve none of the marked points are allowed to coincide, shows that there is also a natural isomorphism
\begin{equation*}
\calC_{0,n}\simeq\calC_{0,\calA_\epsilon}
\end{equation*}
between the universal curves.
\end{example}

\section{Losev-Manin spaces}\label{section_LosevManin}

Let $g=0$. We are now going to consider the special case that the weights $\calA=\{a_0,\ldots,a_{n+1}\}$ for $n\geq 1$ fulfill the two conditions:
\begin{enumerate}[(i)]
\item\label{itemi} $a_0+a_i > 1$ and $a_{n+1}+a_i>1$ for each $i\in\{1,\ldots,n\}$, and 
\item\label{itemii} $a_{i_1}+\ldots +a_{i_r}\leq 1$ for all $\{j_1,\ldots,j_r\}\subseteq\{1,\ldots,n\}$.
\end{enumerate}
We begin with the following easy observation:

\begin{lemma}
For $n\geq 2$ and $g=0$ there is a unique fine chamber in $\scrD_{0,n+2}$ determined by the conditions \eqref{itemi} and \eqref{itemii} above. 
\end{lemma}

\begin{proof}
We have to show that for every $S\subseteq\{0,\ldots, n+1\}$ with $2\leq\vert S\vert \leq n$ the above conditions imply that either $\sum_{i\in S}a_i\leq 1$ or $\sum_{i\in S}a_i>1$.

If $0\notin S$ and $n+1\in S$ or if $0\in S$ and $n+1\notin S$ Condition \eqref{itemi} immediately yields $\sum_{i\in S}Sa_i>1$, since $a_i\geq 0$ for all $i$. In the case that both $0\notin S$ and $n+1\notin S$ we have $\sum_{i\in S}a_i\leq 1$ by condition \eqref{itemii}. Now consider the case $S\supseteq\{0,n+1\}$. By Condition \eqref{itemi} we obtain 
\begin{equation*}
a_0+a_i+a_j+a_{n+1}>2
\end{equation*}
for some $1\leq i,j\leq n$. Since $n\geq 2$ we may assume $i< j$. It follows from Condition \eqref{itemii} that $a_i+a_j\leq 1$ and therefore we obtain $a_0+a_{n+1}>1$, which, in turn, implies
\begin{equation*}
\sum_{i\in S}a_i\geq a_0+a_{n+1}>1 \ ,
\end{equation*}
since $a_i\geq 0$ for all $i$. 

A tuple of weights that fulfills Conditions \eqref{itemi} and \eqref{itemii} is e.g. given by  
\begin{equation*}
\calA=\bigg(1,\frac{1}{n},\ldots,\frac{1}{n},1\bigg) 
\end{equation*}
and therefore this chamber is non-empty. 
\end{proof}

In \cite[Section 6.4]{Hassett_weightedmoduli} Hassett has identified the fine moduli spaces $\Mbar_{0,\calA}$ with the moduli spaces $L_n$, studied by Losev and Manin in \cite{LosevManin_newmoduli}, parametrizing chains $C$ of projective lines connecting the two end points $p_0$ and $p_{n+1}$ with $n$ additional marked points $p_1,\ldots, p_n$ such that
\begin{itemize}
\item the $p_1,\ldots,p_n$ are allowed to mutually coincide, but not to coincide with $p_0$, $p_{n+1}$, or the nodes, and
\item the normalization of every component of $C$ contains at least three special points.
\end{itemize}

By \cite{LosevManin_newmoduli} the moduli space $L_n$ is isomorphic to the smooth projective toric variety defined by the $(n-2)$-dimensional permutohedron $P_{n-1}$ as defined in \cite[Definition 1.3]{Kapranov_permutohedron} and the big torus $T=\G_m^{n-1}$ exactly parametrizes the locus of smooth curves in $L_n$. 

\begin{example}\label{example_L_3nonSNC}

Let us now consider the case $n=3$. We may choose coordinates $(p,q)$ on $T\simeq\G_m^2$ such that up to $\PP^1$-automorphism $(p_0,p_1,p_{4})=(0,1,\infty)$ and $(p,q)=(p_2,p_3)$ are free variables. In these coordinates the toric prime divisors are given by $p=0$, $p=\infty$, $q=0$, $q=\infty$, $p=q=0$, as well as $p=q=\infty$, and the toric boundary has normal crossings. The complement of $\Mbar_{0,5}$ in $L_3$, however, also contains the divisors $p=1$, $q=1$, and $p=q$, which all intersect at $(1,1)$. Therefore $L_3-\Mbar_{0,5}$ does not have normal crossings, as indicated in Remark \ref{remark_nonNC} above.

\begin{center}\begin{tikzpicture}

\draw (-0.5,1.5) -- (1.5,-0.5);
\draw (0,0.5) -- (0,6.5);
\draw (-0.5,6) -- (5.5,6);
\draw (4.5,6.5) -- (6.5,4.5);
\draw (6,5.5) -- (6,-0.5);
\draw (0.5,0) -- (6.5,0);

\draw[dashed] (-0.3,2) -- (6.3,2);
\draw[dashed] (2,-0.3) -- (2,6.3);
\draw[dashed] (0.3,0.3) -- (5.7,5.7);

\node at (3.5,6.3) {$q=\infty$};
\node at (4,-0.3) {$q=0$};
\node[rotate=270] at (6.3,3.5) {$p=\infty$};
\node[rotate=90] at (-0.3,4) {$p=0$};
\node[rotate=-45] at (6,6) {$p=q=\infty$};
\node[rotate=-45] at (0,0) {$p=q=0$};

\node[rotate=90] at (1.7,4) {$p=1$};
\node at (4,1.7) {$q=1$};
\node[rotate=45] at (4.3,3.7) {$p=q$};

\end{tikzpicture}\end{center}

\end{example}

Recall from \cite[Definition 1.3]{Kapranov_permutohedron} that the \emph{$(n-1)$-dimensional permutohedron} $P_n$ is the lattice polytope in $\R^n$ given as the convex hull of the points $\big(s(1),\ldots,s(n)\big)$, where $s$ runs through all elements in the symmetric group $S_n$. As explained in \cite[Definition 2.5.1]{LosevManin_newmoduli} the $l$-dimensional cones of its dual fan $\Delta_n$ are labelled by $(l+1)$-partitions of $\{1,\ldots,n\}$ and therefore naturally carries the structure of a moduli space $L_n^{trop}$ of stable \emph{rational tropical chain curves} with $n+2$ legs connecting the legs $l_0$ and $l_{n+1}$. Its canonical compactification $\overline{L}_n^{trop}$ parametrizes stable \emph{rational extened tropical chain curves} with $n+2$ legs connecting the legs $l_0$ and $l_{n+1}$.

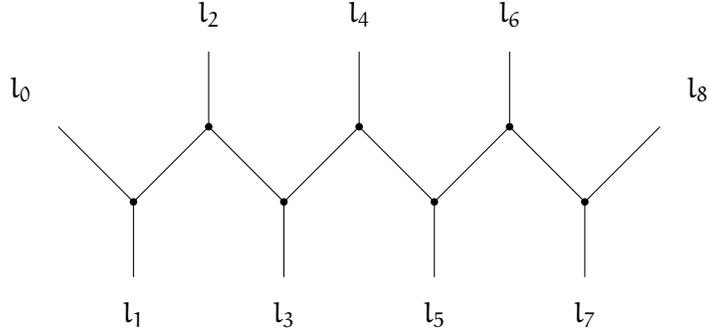
\begin{figure}
\centering\begin{tikzpicture}
\fill (0,1) circle (0.05 cm);
\fill (-1,0) circle (0.05 cm);
\fill (1,0) circle (0.05 cm);
\fill (3,0) circle (0.05 cm);
\fill (2,1) circle (0.05 cm);
\fill (4,1) circle (0.05 cm);
\fill (5,0) circle (0.05 cm);

\draw (-1,0) -- (0,1);
\draw (1,0) -- (0,1);
\draw (0,1) -- (0,2);
\draw (1,0) -- (1,-1);
\draw (-1,0) -- (-1,-1);
\draw (5,0) -- (6,1);
\draw (-1,0) -- (-2,1);
\draw (1,0) -- (2,1);
\draw (2,1) -- (3,0);
\draw (2,1) -- (2,2);
\draw (3,0) -- (3,-1);
\draw (3,0) -- (4,1);
\draw (4,1) -- (5,0);
\draw (4,1) -- (4,2);
\draw (5,0) -- (5,-1);

\node at (0,2.5) {$l_2$};
\node at (6.5,1.5) {$l_8$};
\node at (3,-1.5) {$l_5$};
\node at (-1,-1.5) {$l_1$};
\node at (-2.5,1.5) {$l_0$};
\node at (1,-1.5) {$l_3$};
\node at (2,2.5) {$l_4$};
\node at (4,2.5) {$l_6$};
\node at (5,-1.5) {$l_7$};
\end{tikzpicture}
\caption{A stable rational tropical chain curve with nine legs.}
\end{figure}

\begin{figure}
\centering
\begin{tikzpicture}

\foreach \a in {60,120,180,240,300,360}
\draw (0,0) -- (\a:2);

\foreach \r in {0.33,0.66,1,1.33,1.66}
\foreach \a in {60,120,180,240,300,360}
\draw (\a:\r) -- (\a+60:\r);

\node at (30:2.5) {$(1,2,3)$};
\node at (90:2.5) {$(2,1,3)$};
\node at (150:2.5) {$(2,3,1)$};
\node at (210:2.5) {$(3,2,1)$};
\node at (270:2.5) {$(3,1,2)$};
\node at (330:2.5) {$(1,3,2)$};

\end{tikzpicture}
\caption{The tropical Losev-Manin space $L_3^{trop}$. We indicate the corners of the permutohedron $P_{3}$ dual to $2$-dimensional cones.}
\end{figure}
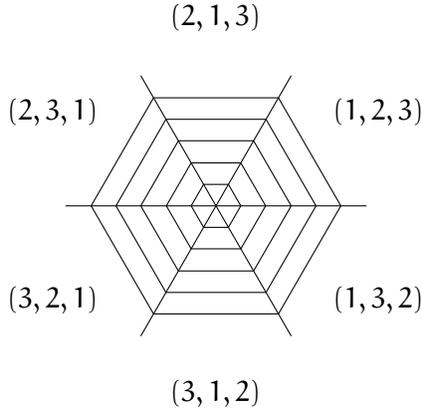

Moreover, there is a natural set-theoretic tropicalization map 
\begin{equation*}
\trop_n\mathrel{\mathop:}L_n^{an}\longrightarrow\overline{L}_n^{trop}
\end{equation*} 
defined analogously to $\trop_{g,\calA}$ in the introduction.

Now recall that Kajiwara \cite[Section 1]{Kajiwara_troptoric} and, independently, Payne \cite[Section 3]{Payne_anallimittrop} construct a tropicalization map 
\begin{equation*}
\trop_{\Delta}\mathrel{\mathop:}X(\Delta)^{an}\longrightarrow N_\R(\Delta)
\end{equation*}
into a partial compactification $N_\R(\Delta)$ of $N_\R$ for all toric varieties $X=X(\Delta)$ defined by a rational polyhedral fan $\Delta$ in $N_\mathbb{R}$, where $N$ denotes the cocharacter lattice of the big torus $T$ of $X$. On an $T$-invariant open affine subset $U=\Spec k[P]$ the partial compactifcation $N_\R(P)$ of $N_\R$ is given by $\Hom(P,\Rbar)$ and $\trop_P$ by
\begin{equation*} \begin{split}
\trop_P\mathrel{\mathop:}U^{an}&\longrightarrow N_\R(P)=\Hom(P,\Rbar)\\
x&\longmapsto \big(p\mapsto -\log\vert p\vert_x\big)
\end{split}\end{equation*}
We also refer the reader to \cite[Section 3]{Rabinoff_newtonpolygon} for further details on this construction.

\begin{corollary}
There is a natural homeomorphism $J_n\mathrel{\mathop:}\overline{L}_n^{trop}\xrightarrow{\sim}N_\R(\Delta_n)$ such that the diagram
\begin{center}\begin{tikzpicture}
  \matrix (m) [matrix of math nodes,row sep=2em,column sep=3em,minimum width=2em]
  {  
  &L_n^{an} & \\ 
 N_\R(\Delta_n)  & &  \overline{L}_n^{trop}  \\ 
  };
  \path[-stealth]
    (m-1-2) edge node [above left] {$\trop_{\Delta_n}$} (m-2-1)
    		edge node [above right] {$\trop_n$} (m-2-3)
    (m-2-3) edge node [below] {$J_n$} node [above] {$\sim$} (m-2-1);		
\end{tikzpicture}\end{center}
commutes. 
\end{corollary}

\begin{proof}
In view of \cite[Theorem 1.2]{Ulirsch_functroplogsch} and \cite[Proposition 7.1]{Ulirsch_functroplogsch} we can naturally identify $\trop_{\Delta_n}$ with the deformation retraction $\mathbf{p}$, since $L_n$ is proper over $k$ and thus the fan $\Delta_n$ is complete. Then the statement immediately follows from Theorem \ref{thm_modulartrop}. 
\end{proof}




\bibliographystyle{amsalpha}
\bibliography{biblio}{}

\def\cprime{$'$}
\providecommand{\bysame}{\leavevmode\hbox to3em{\hrulefill}\thinspace}
\providecommand{\MR}{\relax\ifhmode\unskip\space\fi MR }
\providecommand{\MRhref}[2]{%
  \href{http://www.ams.org/mathscinet-getitem?mr=#1}{#2}
}
\providecommand{\href}[2]{#2}
\begin{thebibliography}{KKMSD73}

\bibitem[ACG11]{ArbarelloCornalbaGriffiths_moduliofcurves}
Enrico Arbarello, Maurizio Cornalba, and Pillip~A. Griffiths, \emph{Geometry of
  algebraic curves. {V}olume {II}}, Grundlehren der Mathematischen
  Wissenschaften [Fundamental Principles of Mathematical Sciences], vol. 268,
  Springer, Heidelberg, 2011, With a contribution by Joseph Daniel Harris.
  \MR{2807457 (2012e:14059)}

\bibitem[ACP12]{AbramovichCaporasoPayne_tropicalmoduli}
Dan Abramovich, Lucia Caporaso, and Sam Payne, \emph{The tropicalization of the
  moduli space of curves}, {arXiv:1212.0373} (2012), Annales de L'ENS, to
  appear.

\bibitem[AG08]{AlexeevGuy_weightedstablemaps}
Valery Alexeev and G.~Michael Guy, \emph{Moduli of weighted stable maps and
  their gravitational descendants}, J. Inst. Math. Jussieu \textbf{7} (2008),
  no.~3, 425--456. \MR{2427420 (2009f:14112)}

\bibitem[BM09]{BayerManin_weightedstablemaps}
Arend Bayer and Yu.~I. Manin, \emph{Stability conditions, wall-crossing and
  weighted {G}romov-{W}itten invariants}, Mosc. Math. J. \textbf{9} (2009),
  no.~1, 3--32, backmatter. \MR{2567394 (2010j:14097)}

\bibitem[BN07]{BakerNorine_RiemannRoch}
Matthew Baker and Serguei Norine, \emph{Riemann-{R}och and {A}bel-{J}acobi
  theory on a finite graph}, Adv. Math. \textbf{215} (2007), no.~2, 766--788.
  \MR{2355607 (2008m:05167)}

\bibitem[BPR11]{BakerPayneRabinoff_nonarchtrop}
Matthew Baker, Sam Payne, and Joseph Rabinoff, \emph{Nonarchimedean geometry,
  tropicalization, and metrics on curves}, {arXiv:1104.0320} (2011), Algebraic
  Geometry, to appear.

\bibitem[Cap13]{Caporaso_tropicalmoduli}
Lucia Caporaso, \emph{Algebraic and tropical curves: comparing their moduli
  spaces}, Handbook of moduli. {V}ol. {I}, Adv. Lect. Math. (ALM), vol.~24,
  Int. Press, Somerville, MA, 2013, pp.~119--160. \MR{3184163}

\bibitem[CHMR14]{CavalieriHampeMarkwigRanganathan_weightedmoduli}
Renzo Cavalieri, Simon Hampe, Hannah Markwig, and Dhruv Ranganathan,
  \emph{Moduli spaces of rational weighted stable curves and tropical
  geometry}, {arXiv:1404.7426} [math] (2014).

\bibitem[CMR14]{CavalieriMarkwigRanganathan_admissiblecovers}
Renzo Cavalieri, Hannah Markwig, and Dhruv Ranganathan, \emph{Tropicalizing the
  space of admissible covers}, {arXiv:1401.4626} [math] (2014), Mathematische
  Annalen, to appear.

\bibitem[CT09]{ConradTemkin_algspaces}
Brian Conrad and Michael Temkin, \emph{Non-archimedean analytification of
  algebraic spaces}, Journal of Algebraic Geometry \textbf{18} (2009), no.~4,
  731–788.

\bibitem[Cue11]{Cueto_geometrictropicalization}
Maria~Angelica Cueto, \emph{Implicitization of surfaces via geometric
  tropicalization}, {arXiv:1105.0509} [math] (2011).

\bibitem[DM69]{DeligneMumford_moduliofcurves}
P.~Deligne and D.~Mumford, \emph{The irreducibility of the space of curves of
  given genus}, Inst. Hautes \'Etudes Sci. Publ. Math. (1969), no.~36, 75--109.
  \MR{0262240 (41 \#6850)}

\bibitem[Fed11]{Fedorchuk_weightedmoduli}
Maksym Fedorchuk, \emph{Moduli of weighted pointed stable curves and log
  canonical models of {$\calMbar_{g,n}$}}, Math. Res. Lett. \textbf{18} (2011),
  no.~4, 663--675. \MR{2831833}

\bibitem[FS13]{FedorchukSmyth_alternatemoduli}
Maksym Fedorchuk and David~Ishii Smyth, \emph{Alternate compactifications of
  moduli spaces of curves}, Handbook of moduli. {V}ol. {I}, Adv. Lect. Math.
  (ALM), vol.~24, Int. Press, Somerville, MA, 2013, pp.~331--413. \MR{3184168}

\bibitem[Has03]{Hassett_weightedmoduli}
Brendan Hassett, \emph{Moduli spaces of weighted pointed stable curves}, Adv.
  Math. \textbf{173} (2003), no.~2, 316--352. \MR{1957831 (2004b:14040)}

\bibitem[HKT09]{HackingKeelTevelev_modulidelPezzo}
Paul Hacking, Sean Keel, and Jenia Tevelev, \emph{Stable pair, tropical, and
  log canonical compactifications of moduli spaces of del {P}ezzo surfaces},
  Invent. Math. \textbf{178} (2009), no.~1, 173--227. \MR{2534095
  (2010i:14062)}

\bibitem[Kaj08]{Kajiwara_troptoric}
Takeshi Kajiwara, \emph{Tropical toric geometry}, Toric topology, Contemp.
  Math., vol. 460, Amer. Math. Soc., Providence, RI, 2008, pp.~197--207.
  \MR{2428356 (2010c:14078)}

\bibitem[Kap93a]{Kapranov_ChowquotientsI}
M.~M. Kapranov, \emph{Chow quotients of {G}rassmannians. {I}}, I. {M}.
  {G}el\cprime fand {S}eminar, Adv. Soviet Math., vol.~16, Amer. Math. Soc.,
  Providence, RI, 1993, pp.~29--110. \MR{1237834 (95g:14053)}

\bibitem[Kap93b]{Kapranov_Veronese&M0nbar}
\bysame, \emph{Veronese curves and {G}rothendieck-{K}nudsen moduli space
  {$\overline M_{0,n}$}}, J. Algebraic Geom. \textbf{2} (1993), no.~2,
  239--262. \MR{1203685 (94a:14024)}

\bibitem[Kap93c]{Kapranov_permutohedron}
Mikhail~M. Kapranov, \emph{The permutoassociahedron, {M}ac {L}ane's coherence
  theorem and asymptotic zones for the {KZ} equation}, J. Pure Appl. Algebra
  \textbf{85} (1993), no.~2, 119--142. \MR{1207505 (94b:52017)}

\bibitem[Kat94]{Kato_toricsing}
Kazuya Kato, \emph{Toric singularities}, Amer. J. Math. \textbf{116} (1994),
  no.~5, 1073--1099. \MR{1296725 (95g:14056)}

\bibitem[KKMSD73]{KKMSD_toroidal}
G.~Kempf, Finn~Faye Knudsen, D.~Mumford, and B.~Saint-Donat, \emph{Toroidal
  embeddings. {I}}, Lecture Notes in Mathematics, Vol. 339, Springer-Verlag,
  Berlin-New York, 1973. \MR{0335518 (49 \#299)}

\bibitem[KM97]{KeelMori_groupoidquotients}
Se{\'a}n Keel and Shigefumi Mori, \emph{Quotients by groupoids}, Ann. of Math.
  (2) \textbf{145} (1997), no.~1, 193--213. \MR{1432041 (97m:14014)}

\bibitem[Knu83]{Knudsen_projectivityII}
Finn~F. Knudsen, \emph{The projectivity of the moduli space of stable curves.
  {II}. {T}he stacks {$M_{g,n}$}}, Math. Scand. \textbf{52} (1983), no.~2,
  161--199. \MR{702953 (85d:14038a)}

\bibitem[LM00]{LosevManin_newmoduli}
A.~Losev and Y.~Manin, \emph{New moduli spaces of pointed curves and pencils of
  flat connections}, Michigan Math. J. \textbf{48} (2000), 443--472, Dedicated
  to William Fulton on the occasion of his 60th birthday. \MR{1786500
  (2002m:14044)}

\bibitem[Man99]{Manin_quantumbook}
Yuri~I. Manin, \emph{Frobenius manifolds, quantum cohomology, and moduli
  spaces}, American Mathematical Society Colloquium Publications, vol.~47,
  American Mathematical Society, Providence, RI, 1999. \MR{1702284
  (2001g:53156)}

\bibitem[MM10]{MustataMustata_weightedstablemaps}
Anca~M. Musta{\c{t}}{\u{a}} and Andrei Musta{\c{t}}{\u{a}}, \emph{Universal
  relations on stable map spaces in genus zero}, Trans. Amer. Math. Soc.
  \textbf{362} (2010), no.~4, 1699--1720. \MR{2574874 (2010k:14102)}

\bibitem[Moo13]{Moon_logcanonicalmodels}
Han-Bom Moon, \emph{Log canonical models for the moduli space of stable pointed
  rational curves}, Proc. Amer. Math. Soc. \textbf{141} (2013), no.~11,
  3771--3785. \MR{3091767}

\bibitem[Ols07]{Olsson_logtwistedcurves}
Martin~C. Olsson, \emph{({L}og) twisted curves}, Compos. Math. \textbf{143}
  (2007), no.~2, 476--494. \MR{2309994 (2008d:14021)}

\bibitem[Pay09]{Payne_anallimittrop}
Sam Payne, \emph{Analytification is the limit of all tropicalizations}, Math.
  Res. Lett. \textbf{16} (2009), no.~3, 543--556. \MR{2511632 (2010j:14104)}

\bibitem[Rab12]{Rabinoff_newtonpolygon}
Joseph Rabinoff, \emph{Tropical analytic geometry, {N}ewton polygons, and
  tropical intersections}, Adv. Math. \textbf{229} (2012), no.~6, 3192--3255.
  \MR{2900439}

\bibitem[Ran89]{Ran_defofmaps}
Ziv Ran, \emph{Deformations of maps}, Algebraic curves and projective geometry
  ({T}rento, 1988), Lecture Notes in Math., vol. 1389, Springer, Berlin, 1989,
  pp.~246--253. \MR{1023402 (91f:32021)}

\bibitem[Ran15]{Ranganathan_ratcurvesnonArch}
Dhruv Ranganathan, \emph{Moduli of rational curves in toric varieties and
  non-{Archimedean} geometry}, arXiv:1506.03754 [math] (2015).

\bibitem[Thu07]{Thuillier_toroidal}
Amaury Thuillier, \emph{G\'eom\'etrie toro\"\i dale et g\'eom\'etrie analytique
  non archim\'edienne. {A}pplication au type d'homotopie de certains sch\'emas
  formels}, Manuscripta Math. \textbf{123} (2007), no.~4, 381--451. \MR{2320738
  (2008g:14038)}

\bibitem[Uli13]{Ulirsch_functroplogsch}
Martin Ulirsch, \emph{Functorial tropicalization of logarithmic schemes: The
  case of constant coefficients}, {arXiv:1310.6269} [math] (2013).

\bibitem[Uli14]{Ulirsch_nonArchstacks}
\bysame, \emph{A geometric theory of non-{Archimedean} analytic stacks},
  arXiv:1410.2216 [math] (2014).

\bibitem[Viv13]{Viviani_tropcompTorelli}
Filippo Viviani, \emph{Tropicalizing vs. compactifying the {T}orelli morphism},
  Tropical and non-{A}rchimedean geometry, Contemp. Math., vol. 605, Amer.
  Math. Soc., Providence, RI, 2013, pp.~181--210. \MR{3204272}

\end{thebibliography}

\end{document}